\documentclass[10pt]{amsart}
\usepackage{amsthm}
\usepackage{amssymb}
\usepackage{amsmath}
\input xy
\xyoption{all}
\usepackage{amsthm}
\usepackage{amssymb}
\usepackage{amsmath}
\usepackage[shortlabels]{enumitem}

\usepackage{ifpdf}
\ifpdf
\usepackage[pdftex,linktocpage]{hyperref}
\else
\usepackage[hypertex,linktocpage]{hyperref}
\fi

\input xy
\xyoption{all} %\pagestyle{plain}

\newtheorem{theorem}{Theorem}
\newtheorem*{theorem*}{Theorem}

\newtheorem{definition}[theorem]{Definition}
\newtheorem*{definition*}{Definition}
\newtheorem{lemma}[theorem]{Lemma}

\newtheorem{proposition}[theorem]{Proposition}
\newtheorem{remark}[theorem]{Remark}
\newtheorem{cor}[theorem]{Corollary}
\newtheorem*{cor*}{Corollary}

\newtheorem{question}{Question}

\newtheorem*{conjecture*}{Conjecture}
\numberwithin{theorem}{section}

\renewcommand\iff{%
\ifmmode\text{ if and only if }%
\else if and only if \fi}

\renewcommand{\and}{\wedge}

\renewcommand{\phi}{\varphi}

\newcommand{\Mod}{\textnormal{Mod}}
\renewcommand{\mod}{\text{mod}}

\newcommand{\Ab}{\text{Ab}}

\newcommand{\rad}{\textnormal{rad}}

\newcommand{\End}{\textnormal{End}}
\newcommand{\Hom}{\textnormal{Hom}}
\newcommand{\Ext}{\textnormal{Ext}}
\newcommand{\ann}{\textnormal{ann}}

\newcommand{\zg}{\textnormal{Zg}}
\newcommand{\Zg}{\textnormal{Zg}}

\newcommand{\pinj}{\textnormal{pinj}}

\newcommand{\mcal}[1]{\mathcal{#1}}

\newcommand{\st}{\ \vert \ }

\newcommand{\pp}{\textnormal{pp}}

\newcommand{\Q}{\mathbb{Q}}
\newcommand{\N}{\mathbb{N}}
\newcommand{\Z}{\mathbb{Z}}

\newcommand{\Tf}{\text{Tf}}
\newcommand{\Latt}{\text{Latt}}

\newcommand\hfuzzReset{\hfuzz=3pt}
\hfuzzReset
\newcommand\toleranceReset{\tolerance=1400}
\toleranceReset
\newcommand\emergencystretchReset{\emergencystretch=2ex}% 2.5ex is too wide
\emergencystretchReset
\title{Maranda's Theorem for Pure-Injective Modules and Duality}
\author{Lorna Gregory}
\address{Dipartimento di Matematica e Fisica, Universit\'a degli Studi della Campania ``Luigi
Vanvitelli'', Viale Lincoln 5, 81100 Caserta, Italy}
\email{Lorna.Gregory@gmail.com}

\thanks{The majority of this work was completed while the author employed by the University of Camerino.}



\begin{document}

\begin{abstract}
Let $R$ be a discrete valuation domain with field of fractions $Q$ and maximal ideal generated by $\pi$. Let $\Lambda$ be an $R$-order such that $Q\Lambda$ is a separable $Q$-algebra. Maranda showed that there exists $k\in\N$ such that for all $\Lambda$-lattices $L$ and $M$, if $L/L\pi^k\simeq M/M\pi^k$ then $L\simeq M$. Moreover, if $R$ is complete and  $L$ is an indecomposable $\Lambda$-lattice, then $L/L\pi^k$ is also indecomposable. We extend Maranda's theorem to the class of $R$-reduced $R$-torsion-free pure-injective $\Lambda$-modules.

As an application of this extension, we show that if $\Lambda$ is an order over a Dedekind domain $R$ with field of fractions $Q$ such that $Q\Lambda$ is separable then the lattice of open subsets of the $R$-torsion-free part of the right Ziegler spectrum of $\Lambda$ is isomorphic to the lattice of open subsets of the $R$-torsion-free part of the left Ziegler spectrum of $\Lambda$.

Finally, with $k$ as in Maranda's theorem, we show that if $M$ is $R$-torsion-free and $H(M)$ is the pure-injective hull of $M$ then $H(M)/H(M)\pi^k$ is the pure-injective hull of $M/M\pi^k$. We use this result to give a characterisation of $R$-torsion-free pure-injective $\Lambda$-modules and describe the pure-injective hulls of certain $R$-torsion-free $\Lambda$-modules.

%ATTEMPT 1:
%Maranda's theorem for lattices over orders states that, if $R$ is a discrete valuation domain with field of fractions $Q$ and maximal ideal generated by $\pi$ and $\Lambda$ an $R$-order inside a separable $Q$-algebra, then there exists a natural number $k$ such that for all $\Lambda$-lattices $L$ and $M$, if $L/L\pi^k\simeq M/M\pi^k$ then $L\simeq M$. Moreover, if $R$ is complete and a $\Lambda$-lattice $L$ is indecomposable, then $L/L\pi^k$ is also indecomposable. We extend Maranda's theorem to the class of $R$-reduced $R$-torsionfree pure-injective $\Lambda$-modules. Moreover, with $k$ as in Maranda's theorem we show that BLAH preserves pure-injective hulls.
%As an application of this extension, we show that if $\Lambda$ is an order over a Dedekind domain then the lattice of open subsets of the $R$-torsionfree part of the right Ziegler spectrum of $\Lambda$ is isomorphic to the lattice of open subsets of the $R$-torsionfree part of the left Ziegler spectrum of $\Lambda$.
\end{abstract}

\maketitle

%{
%\renewcommand{\thetheorem}{\ref{marandapresisotype}}
%\begin{theorem}
%  Lorem ipsum ...
%\end{theorem}
%\addtocounter{theorem}{-1}
%}
%
%{
%\renewcommand{\thetheorem}{\ref{marandapresind}}
%\begin{theorem}
%  Lorem ipsum ...
%\end{theorem}
%\addtocounter{theorem}{-1}
%}
%
%{
%\renewcommand{\thetheorem}{\ref{prespihulls}}
%\begin{theorem}
%  Lorem ipsum ...
%\end{theorem}
%\addtocounter{theorem}{-1}
%}

\section{Introduction}

Let $\Lambda$ be an order over a discrete valuation domain $R$ with maximal ideal generated by $\pi$ and field of fractions $Q$ such that $Q\Lambda$ is a separable $Q$-algebra. Maranda's theorem states that there exists $k_0\in\N$ such that for all $k\geq k_0+1$ and $\Lambda$-lattices $L,M$, $L/L\pi^k\cong M/M\pi^k$ implies $L\cong M$ and if $R$ is complete then $L$ indecomposable implies $L/L\pi^k$ is indecomposable.

In this paper we extend Maranda's theorem to a class of pure-injective modules over $\Lambda$. From our perspective at least, the natural non-finitely-presented generalisation of a $\Lambda$-lattice is an $R$-torsion-free $\Lambda$-module. This is because the smallest definable subcategory of $\Mod\text{-}\Lambda$ containing all (right) $\Lambda$-lattices is exactly the category, $\Tf_\Lambda$, of (right) $R$-torsion-free $\Lambda$-modules. We write ${_\Lambda}\Tf$ for the category of $R$-torsion-free left $\Lambda$-modules. An alternative non-finitely-presented version of a $\Lambda$-lattice, the generalised lattices, was introduced in \cite{GenLattButler} and further studied in \cite{LargeLattRump} and \cite{GenLattPrihPun}.

Every $R$-torsion-free $\Lambda$-module decomposes as a direct sum $D\oplus N$ of an $R$-divisible module $D$ and an $R$-reduced module $N$ i.e. $\bigcap_{i\in\N}N\pi^i=0$. If $D$ is divisible then $D/D\pi^k=0$. Thus we restrict our generalisation of Maranda's theorem further to the class of $R$-reduced $R$-torsion-free pure-injective $\Lambda$-modules.

In section \ref{SMarfun}, with $k_0$ as in the classical version of Maranda's theorem, we prove the following theorems.

\medskip

\noindent
{\bfseries Theorem 3.4.} \textit{Let $M,N$ be $R$-torsion-free $R$-reduced pure-injective $\Lambda$-modules. If $M/M\pi^k\cong N/N\pi^k$ for some $k\geq k_0+1$ then $M\cong N$.}

\medskip

\noindent
{\bfseries Theorem 3.5.} \textit{Let $k\geq k_0+1$. If $N$ is an indecomposable $R$-torsion-free $R$-reduced pure-injective $\Lambda$-module then $N/N\pi^k$ is indecomposable.}

\medskip

Note that $\Lambda$-lattices are pure-injective if and only if $R$ is complete. So the fact that we do not need to assume that $R$ is complete is not unexpected.

%\begin{theorem}\label{marandapresisotype}
%Let $M,N$ be $R$-torsion-free $R$-reduced pure-injective $\Lambda$-modules. If $M/M\pi^k\cong N/N\pi^k$ for some $k\geq k_0+1$ then $M\cong N$.
%\end{theorem}

Using results from \cite{TfpartZgorders}, we get the following.

\medskip

\noindent
{\bfseries Theorem 3.8.} \textit{Let $k\geq k_0+1$. Suppose that $M$ is $R$-torsion-free and $R$-reduced. If $u:M\rightarrow H(M)$ is the pure-injective hull of $M$ then $\overline{u}:M/M\pi^k\rightarrow H(M)/H(M)\pi^k$ is the pure-injective hull of $M/M\pi^k$.}

\medskip

The modules $M/M\pi^k$ may be naturally viewed as modules over the Artin algebra $\Lambda/\Lambda\pi^k$. Our proofs of these theorems and their applications rely on the fact that the functor taking $M\in\Tf_\Lambda$ to $M/M\pi^k\in\Mod\text{-}\Lambda/\Lambda\pi^k$, which, for $k$ sufficiently large, we will refer to as \textbf{Maranda's functor}, is an interpretation functor. The original definition (see section \ref{S-prelim}) of an interpretation functor came out of the model theoretic notion of an interpretation. However, from an algebraic perspective, interpretation functors are just additive functors which commutes with direct limits and direct products.

Thanks to Maranda's theorem, in theory, in order to get information about the category of $\Lambda$-lattices we may instead study a subcategory of the category of modules over the Artin algebra $\Lambda/\Lambda\pi^k$. The drawback of both the classical version of Maranda's theorem and our extended version is that $\mod\text{-}\Lambda_k$, respectively $\Mod\text{-}\Lambda_k$, is almost always significantly more complicated than the category of $\Lambda$-lattices, respectively $\Tf_\Lambda$. For instance, the order $\Z_{(p)}C(p^2)$ is finite lattice type (see \cite{Butlerintreps}). The category of $\Z_{(p)}/p^2\Z_{(p)}$-free finitely generated $\Z_{(p)}/p^2\Z_{(p)}C(p^2)$-modules is wild \cite{Bondarenko}.

Despite the above, we will see in sections \ref{S-Piandpihulls} and \ref{S-duality} that being able to move from $\Tf_\Lambda$ to a module category over an Artin algebra has useful applications. Sections \ref{S-Piandpihulls} and \ref{S-duality}, are largely independent of each other.

Section \ref{S-Piandpihulls} presents applications of \ref{prespihulls} to pure-injectives and pure-injective hulls in $\Tf_{\Lambda}$. We give the following characterisation of pure-injective $R$-torsion-free $\Lambda$-modules. %when $R$ is a discrete valuation domain.

\medskip

\noindent
{\bfseries Theorem 4.6.} \textit{Let $M\in\Tf_\Lambda$. Then $M$ is pure-injective if and only if
\begin{enumerate}
\item $M/M\pi^k$ is pure-injective for all $k\in\N$ and
\item $M$ is pure-injective as an $R$-module.
\end{enumerate}}

\medskip

We also give information about the pure-injective hull of an $R$-reduced $R$-torsion-free module $M$ in terms of pure-injective hulls of $M/M\pi^k$ for all $k\geq k_0+1$. In particular, when $M$ is reduced, $R$-torsion-free and $M/M\pi^k$ is pure-injective for all $k\in \N$, we show, \ref{pihullquotientspi}, that the pure-injective hull of $M$ is the inverse limit of the $\Lambda$-modules $M/M\pi^k$ along the canonical projections.

The right Ziegler spectrum $\Zg_S$ of a ring $S$ is a topological space which captures the majority of model theoretic information about $\Mod\text{-}S$. The points of $\Zg_S$ are (isomorphism classes of) indecomposable pure-injective right $S$-modules and its closed sets correspond to definable subcategories of $\Mod\text{-}S$. If $\Lambda$ is an $R$-order then $\Zg_\Lambda^{tf}$, the torsion-free part of the Ziegler spectrum of $\Lambda$, is the closed set of indecomposable pure-injective modules which are $R$-torsion-free. This space was studied for $RG$ where $G$ is a finite group and $R$ is a Dedekind domain in \cite{TFpartZgRG}, for the $\widehat{\Z_{(2)}}$-order $\widehat{\Z_{(2)}}C_2\times C_2$ in \cite{TFpartkleinfour} and more recently, for general orders over Dedekind domains, in \cite{TfpartZgorders}. We will write ${_S}\Zg$ for the left Ziegler spectrum of $S$ and ${_\Lambda}\Zg^{tf}$ for the torsion-free part of the left Ziegler spectrum of $\Lambda$.

The space $\Zg_\Lambda^{tf}$ was described topologically for the $\widehat{\Z_{(2)}}$-order $\widehat{\Z_{(2)}}C_2\times C_2$ but not all the points were described as modules. As a practical application of our results on pure-injective hulls, with the help of results of Krause on generalised tubes in categories of modules over Artin algebras in \cite{KrauseGeneric}, we are able to describe, as $\widehat{\Z_{(2)}}C_2\times C_2$-modules, the pure-injective hulls of the Pr\"{u}fer like modules, denoted $T$ in \cite{TFpartkleinfour}. Moreover, using results of Butler \cite{ButlerKleinfour}, Dieterich \cite{ARquiversDieterich}, and Puninski and Toffalori \cite{TFpartkleinfour}, we show that the pure-injective hulls of these modules are indecomposable and thus are points of $\Zg^{tf}_\Lambda$. As far as we are aware, until now, the only points of $\Zg_\Lambda^{tf}$ for any order $\Lambda$ which have been explicitly described as modules are $\widehat{\Lambda}$-lattices, where $\widehat{R}$ is he completion of $R$ and $\widehat{\Lambda}:=\widehat{R}\otimes\Lambda$, and the $R$-divisible modules.

The theme of section \ref{S-duality} is links between $\Tf_\Lambda$ and ${_\Lambda}\Tf$. Here we extend our setting to included the case where $R$ is a Dedekind domain with field of fractions $Q$ and $\Lambda$ is an $R$-order such that $Q\Lambda$ is a separable $Q$-algebra.

Ivo Herzog \cite{Herzogduality} showed that for any ring $S$, the lattice of open subsets of $\Zg_S$ and the lattice of open subsets of ${_S}\Zg$ are isomorphic. Applying Herzog's result directly to $\Zg_\Lambda$, shows that the lattice of open subsets of $\Zg_\Lambda^{tf}$ is isomorphic to the lattice of open subsets of the closed subset of $R$-divisible modules in ${_\Lambda}\Zg$.  However, using our extended version of Maranda's theorem to move to $\Zg_{\Lambda/\Lambda\pi^k}$ then applying Herzog's duality allows us to prove, \ref{dualitymain}, that the lattice of open subsets of $\Zg^{tf}_\Lambda$ is isomorphic to the lattice of open subsets of ${_\Lambda}\Zg^{tf}$.

For $S$ a ring, the Krull-Gabriel dimension of $(\mod\text{-}S,\Ab)^{fp}$ the category of finitely presented additive functors from $\mod\text{-}S$, the category of finitely presented right $S$-modules, to $\Ab$, the category of abelian groups, is an ordinal valued measure of the complexity of $\Mod\text{-}S$. The Krull-Gabriel dimension of $(\mod\text{-}S,\Ab)^{fp}$ (respectively $(S\text{-}\mod,\Ab)^{fp}$) is equal to the m-dimension of the lattice of right (respectively left) pp formulas of $S$ (see \cite[13.2.2]{PSL}). Moreover, \cite[7.2.4]{PSL}, the m-dimension of the lattice of right pp formulas of $S$ is equal to the m-dimension of the lattice of left pp formulas of $S$.

Using results in \cite{TfpartZgorders}, we show, \ref{mdimdual}, that the m-dimension of the lattice of (right) pp formulas of $\Lambda$ with respect to the theory of $\Tf_{\Lambda}$ is equal to the m-dimension of the lattice of (left) pp formulas of $\Lambda$ with respect to the theory of ${_\Lambda}\Tf$. As a consequence, we show, \ref{KGdual}, that the Krull-Gabriel dimension of $(\Latt_{\Lambda},\Ab)^{fp}$ is equal to the Krull-Gabriel dimension of $({_\Lambda}\Latt,\Ab)^{fp}$ where $\Latt_{\Lambda}$ is the category of right $\Lambda$-lattices and ${_\Lambda}\Latt$ is the category of left $\Lambda$-lattices.

Before starting the main body of the paper, the reader should be \textit{warned} that the word \textit{lattice} has two meanings in this paper; the first, a particular type of $\Lambda$-module and the second a partially ordered set with meets and joins. Since these objects are so different in character, it shouldn't cause confusion.

\section{Preliminaries}\label{S-prelim}

We start by introducing some notation and basic definitions relating to orders. For a general introduction to orders and their categories of lattices we suggest \cite{CurtisReinerV1}.

Let $R$ be a Dedekind domain. An $R$-order $\Lambda$ is an $R$-algebra which is finitely generated and $R$-torsion-free as an $R$-module. A $\Lambda$-lattice is a finitely generated $\Lambda$-module which is $R$-torsion-free. We will write $\Latt_\Lambda$ (respectively ${_\Lambda}\Latt$) for the category of right (respectively left) $\Lambda$-lattices and $\Tf_\Lambda$ (respectively ${_\Lambda}\Tf$) for the category of right (respectively left) $R$-torsion-free modules.

Let $\text{Max}R$ denote the set of prime ideals of $R$. If $P\in\text{Max}R$ then $\Lambda_P$, the localisation of $\Lambda$ at the multiplicative set $R\backslash P$, is an $R_P$-order. Let $\widehat{R_P}$ and $\widehat{\Lambda_P}$ denote the $P$-adic completions of $R_P$ and $\Lambda_P$ respectively. Note that $\widehat{\Lambda_P}$ is an $\widehat{R_P}$-order. If $L\in \Latt_\Lambda$  and $P\in\text{Max}R$ then $L_P$ will denote $R_P\otimes_RL$. If $L\in\Latt_\Lambda$ then $\widehat{L_P}$ will denote the $P$-adic completion of $L$. Note that if $L\in\Latt_\Lambda$ then $L_P$ is a $\Lambda_P$-lattice and $\widehat{L_P}$ is a $\widehat{\Lambda_P}$-lattice.

We now give a summary of the notions from model theory of modules that will be used in this paper. For a more detailed introduction the reader is referred to \cite{MikesBluebook} and \cite{PSL}.

We will write $\mathbf{x}$ for tuples of variables and likewise $\mathbf{m}$ for tuples of elements in a module.

Let $S$ be a ring. A (right) \textbf{pp-$n$-formula} is a formula in the language of $S$-modules of the form \[\exists \mathbf{y} \ (\mathbf{y},\mathbf{x})A=0\] where $A$ is an $(l+n)\times m$ matrix with entries from $S$, $\mathbf{y}$ is an $l$-tuple of variables, $\mathbf{x}$ is an $n$-tuple of variables and  $l,n,m$ are natural numbers.

If $M\in\Mod\text{-}S$ then we write $\phi(M)$ for the \textbf{solution set} of $\phi$ in $M$. For any pp-$n$-formula $\phi$ and $S$-module $M$, $\phi(M)$ is a $\End(M)$-submodule of $M^n$ under the diagonal action of $\End(M)$ on $M^n$.

After identifying (right) pp-$n$-formulas $\phi,\psi$ such that $\phi(M)=\psi(M)$ for all $M\in\Mod\text{-}S$, the set of pp-$n$-formulas becomes a lattice under inclusion of solution sets i.e. $\psi\leq \phi$ if $\psi(M)\subseteq \psi(M)$ for all $M\in \Mod\text{-}S$. We denote this lattice by $\pp_S^n$ and the left module version by ${_S}\pp^n$. If $X$ is a collection of (right) $S$-modules then we write $\pp_S^nX$ for the quotient of $\pp_S^n$ under the equivalence relation $\phi\sim_{X}\psi$ if $\phi(M)=\psi(M)$ for all $M\in X$.

A \textbf{pp-$n$-pair}, written $\phi\, / \,\psi$, is a pair of pp-$n$-formulas $\phi,\psi$ such that $\phi(M)\supseteq \psi(M)$ for all $S$-modules $M$. If $\phi\, / \,\psi$ is a pp-$n$-pair then we write $[\psi,\phi]$ for the interval in $\pp_R^n$, that is, the set of $\sigma\in\pp_S^n$ such that $\psi\leq \sigma\leq \phi$. If $X$ is a collection of (right) $S$-modules, we will write $[\psi,\phi]_X$ for the corresponding interval in $\pp_S^nX$.

If $\mathbf{m}$ is an $n$-tuple of elements from a module $M$ then the \textbf{pp-type} of $\mathbf{m}$ is the set of pp-$n$-formulas $\phi$ such that $\mathbf{m}\in\phi(M)$. If $M\in \mod\text{-}R$ and $\mathbf{m}$ is an $n$-tuple of element from $M$ then, \cite[1.2.6]{PSL}, there exists $\phi\in\pp_R^n$ such that $\psi$ is in the pp-type of $\mathbf{m}$ if and only if $\psi\geq \phi$.

%Let $\phi\in\pp_R^n$. A \textbf{free realisation} of $\phi$ is a pair $(C,\overline{c})$ where $C\in\mod\text{-}R$ and $\mathbf{c}$ is an $n$-tuple of elements from $C$ with the property that the pp-type of $\mathbf{c}$ in $C$ is generated by $\phi$ i.e. $\psi(\overline{c})$ holds in $C$ if and only if $\phi\leq \psi$.
%
%
%\begin{proposition}\cite[1.2.14,1.2.7]{PSL}
%Every pp formula $\phi$ has a free realisation. Moreover, if $(C,\mathbf{c})$ is a free realisation for $\phi$ and $\mathbf{m}\in\phi(M)$ for some module $M$ and tuple $\mathbf{m}$ of elements from $M$ then there is a homomorphism $f:C\rightarrow M$ such that $f(\mathbf{c})=\mathbf{m}$.
%\end{proposition}

For each $n\in\N$, Prest defined a lattice anti-isomorphism $D:\pp_S^n\rightarrow {_S}\pp^n$ (see \cite[section 1.3.1]{PSL} and \cite[8.21]{MikesBluebook}). As is standard, we denote its inverse ${_S}\pp^n\rightarrow \pp_S^n$ also by $D$. Apart from the fact that for $a\in S$, $D(xa=0)$ is $a|x$ and $D(a|x)$ is $ax=0$, we will not need to explicitly take the dual of a pp formula here, so we will not give its definition.

An embedding $f:M\rightarrow N$ is a \textbf{pure-embedding} if for all $\phi\in\pp_S^1$, $\phi(N)\cap f(M)=f(\phi(M))$. Equivalently, for all $L\in S\text{-}\mod$, $f\otimes -: M\otimes L\rightarrow N\otimes L$ is an embedding. We say $N$ is \textbf{pure-injective} if every pure-embedding $g:N\rightarrow M$ is a split embedding. Equivalently, $N$ is pure-injective if and only if it is \textbf{algebraically compact} \cite[4.3.11]{PSL}. That is, for all $n\in\N$, if for each $i\in \mcal{I}$, $\mathbf{a_i}\in N$ is an $n$-tuple and $\phi_i$ is a pp-$n$-formula then $\bigcap_{i\in \mcal{I}}\mathbf{a_i}+\phi_i(N)=\emptyset$ implies there is some finite subset $\mcal{I}'$ of $\mcal{I}$ with $\bigcap_{i\in \mcal{I}'}\mathbf{a_i}+\phi_i(N)=\emptyset$.

We will write $\pinj_S$ (respectively ${_S}\pinj$) for the set of (isomorphism types of) indecomposable pure-injective right (respectively left) $S$-modules.

We say a pure-embedding $i:M\rightarrow N$ with $N$ pure-injective is a \textbf{pure-injective hull} of $M$ if for every other pure-embedding $g:M\rightarrow K$ where $K$ is pure-injective, there is a pure-embedding $h:N\rightarrow K$ such that $hf=g$. The pure-injective hull of $M$ is unique up to isomorphism over $M$ and we will write $H(M)$ for any module $N$ such that the inclusion of $M$ in $N$ is a pure-injective hull of $M$.

The following lemma will be used in section \ref{S-duality}. Its proof is exactly as in \cite[3.1]{TFpartZgRG}.

\begin{lemma}\label{pihulllattices}
Let $M$ be a $\Lambda$-lattice. The pure-injective hull of $M$ is isomorphic to $\prod_{P\in\text{Max}R}\widehat{M_P}$.
\end{lemma}

A full subcategory of a module category $\Mod\text{-}S$ is a \textbf{definable subcategory} if it satisfies the equivalent conditions in the following theorem.
\goodbreak
\begin{theorem}\cite[3.4.7]{PSL}\label{defsubcatdef} The following statements are equivalent for $\mcal{X}$ a full subcategory of $\Mod\text{-}S$.
\begin{enumerate}
\item There exists a set of pp-pairs $\{\phi_i/\psi_i\st i\in I\}$ such that $M\in\mcal{X}$ if and only $\phi_i(M)=\psi_i(M)$ for all $i\in I$.
\item $\mcal{X}$ is closed under direct products, direct limits and pure submodules.
\item $\mcal{X}$ is closed under direct products, reduced products and pure submodules.
\item $\mcal{X}$ is closed under direct products, ultrapowers and pure submodules.
\end{enumerate}
\end{theorem}

For an $R$-order $\Lambda$, a particularly important definable subcategory is, $\Tf_\Lambda$, the class of all $R$-torsion-free $\Lambda$-modules. It is the class of $\Lambda$-modules such that for all $r\in R$, the solution set of $xr=0$ in $M$ is equal to the solution set of $x=0$ in $M$.

Given a class of modules $\mcal{C}$, let $\langle \mcal{C}\rangle $ denote the smallest definable subcategory containing $\mcal{C}$. Since all modules in $\Tf_\Lambda$ are direct unions of their finitely generated submodules and a finitely generated $R$-torsion-free module is a $\Lambda$-lattice, $\langle \Latt_\Lambda\rangle=\Tf_\Lambda$.

If $\mcal{C}\subseteq \Mod\text{-}S$ then we will write $\pinj(\mcal{C})$ for the set of (isomorphism types of) indecomposable pure-injective $S$-modules contained in $\mcal{C}$. By \cite[5.1.4]{PSL}, definable subcategories of $\Mod\text{-}S$ are determined by the indecomposable pure-injective $S$-modules they contain i.e. $\mcal{C}=\langle\pinj(\mcal{C})\rangle$.

The (right) \textbf{Ziegler spectrum} of a ring $S$, denoted $\zg_S$, is a topological space whose points are isomorphism classes of indecomposable pure-injective (right) $S$-modules and which has a basis of open sets given by:
\[(\phi\, / \,\psi)=\{M\in\pinj_S \st \phi(M)\supsetneq\psi(M)\and\phi(M)\}\]
where $\varphi,\psi$ range over (right) pp-$1$-formulas. We write ${_S}\Zg$ for the left Ziegler spectrum of $S$.

The sets $(\phi\, / \,\psi)$ are compact, in particular, $\Zg_S$ is compact.

From (i) of \ref{defsubcatdef}, it is clear that if $\mcal{X}$ is a definable subcategory of $\Mod\text{-}S$ then $\mcal{X}\cap \text{pinj}_S$ is a closed subset of $\Zg_S$ and that all closed subsets of $\Zg_S$ arise in this way. Since definable subcategories are determined by the indecomposable pure-injective modules they contain, if $\mcal{X},\mcal{Y}$ definable subcategories of $\Mod\text{-}S$, then $\mcal{X}\cap \Zg_S=\mcal{Y}\cap \Zg_S$ if and only if $\mcal{X}=\mcal{Y}$. Thus there is an inclusion preserving correspondence between the closed subsets of $\Zg_S$ and the definable subcategories of $\Mod\text{-}S$. If $\mcal{X}$ is a definable subcategory of $\Mod\text{-}S$ then we will write $\Zg(\mcal{X})$ for the Ziegler spectrum of $\mcal{X}$, that is, $\mcal{X}\cap\Zg_S$ with the topology inherited from $\Zg_S$. When $\Lambda$ is an $R$-order, we will write $\Zg_\Lambda^{tf}$ (respectively ${_\Lambda}\Zg^{tf}$) for $\Zg(\Tf_\Lambda)$ (respectively $\Zg({_\Lambda}\Tf)$).

We finish this section by introducing interpretation functors and proving a result about them which we will need in section \ref{S-duality}.

Let $\mcal{C}\subseteq \Mod\text{-}S$ and $\mcal{D}\subseteq \Mod\text{-}T$ be definable subcategories.  Let $\phi/\psi$ be a pp-$m$-pair over $S$ and for each $t\in T$, let $\rho_t(\overline{x},\overline{y})$ be a pp-$2m$-formula such that for each $M\in\mcal{C}$, the solution set $\rho_t(M,M)\subseteq M^m\times M^m$ defines an endomorphism $\rho_t^M$ of the abelian group $\phi(M)/\psi(M)$ and such that $\phi(M)/\psi(M)$ is an $S$-module in $\mcal{D}$ when for all $t\in T$, the action of $t$ on $\phi(M)/\psi(M)$ is given by $\rho^M_t$. In this situation $(\phi/\psi;(\rho_t)_{t\in T})$ defines an additive functor $I:\mcal{C}\rightarrow \mcal{D}$. Following \cite{Interpretingmodules}, we call any functor equivalent to one defined in this way an \textbf{interpretation functor}.

From the definition it is clear that for $k\in\N$, the functor $I:\Tf_\Lambda\rightarrow \Mod\text{-}\Lambda/\pi^k\Lambda$ which send $M\in\Tf_\Lambda$ to $M/M\pi^k$ is an interpretation functor. We will consider another interpretation functor, Butler's functor, at the end of section \ref{S-Piandpihulls}.

The following theorem, due to Prest in full generality and Krause in a special case, gives a completely algebraic characterisation of interpretation functors.

\begin{theorem}\cite[25.3]{Defaddcats}\cite[7.2]{Exdefcat}
An additive functor $I:\mcal{C}\rightarrow \mcal{D}$ is an interpretation functor if and only if it commutes with direct products and direct limits.
\end{theorem}

Define $\ker I$ to be the definable subcategory of objects $L\in\mcal{C}$ such that $IL=0$. For $\mcal{D}'$ a definable subcategory of $\mcal{D}$, let $I^{-1}\mcal{D}'$ be the definable subcategory of objects $L\in\mcal{C}$ such that $IL\in\mcal{D}'$.

The following lemma is used in various places in the literature. It follows easily from (ii) of \ref{defsubcatdef}.

\begin{lemma}\label{imintfunctorpuresubmod}
Let $I:\mcal{C}\rightarrow\mcal{D}$ be an interpretation functor and $\mcal{C}'$ a definable subcategory of $\mcal{C}$. Then the closure of $I\mcal{C}'$ under pure-subobjects is a definable subcategory of $\mcal{D}$.
\end{lemma}
%\begin{proof}
%Since $I$ commutes with products and reduced products, $I\mcal{C}$ is closed under products and reduced products. Thus, it is enough to note that if $L_i\in\mcal{C}$ for all $i\in I$ and $M_i$ is a pure-subobject of $L_i$ for each $i\in I$, then $\prod M_i$ is a pure-subobject of $\prod L_i$ and for all filters $\mcal{F}$ on $I$, $\prod M_i/\mcal{F}$ is a pure-subobject of $\prod L_i/\mcal{F}$.
%\end{proof}

%\begin{proposition}\cite[18.2.24]{PSL}\label{pinjdirectsums}
%Let $I:\mcal{C}\rightarrow \mcal{D}$ be an interpretation functor between definable categories $\mcal{C}$ and $\mcal{D}$. Let $\mcal{C}'$ be a definable subcategory of $\mcal{C}$ and suppose that $N$ is an indecomposable pure-injective summand of $IC$ for some $C\in\mcal{C}'$. Then, there exists $N'\in \mcal{C}'$ indecomposable pure-injective such that $N$ is a direct summand of $IN'$.
%\end{proposition}

\begin{lemma}\label{defsubcatunderIinv}
Let $I:\mcal{C}\rightarrow \mcal{D}$ be an interpretation functor such that for all $N\in \pinj(\mcal{C})$, $IN=0$ or $IN\in\pinj(\mcal{D})$ and if $N,M\in \pinj(\mcal{C})$, $IN,IM\neq 0$ and $IN\cong IM$ then $N\cong M$.
\begin{enumerate}
\item If $\mcal{C}'$ is a definable subcategory of $\mcal{C}$ containing $\ker I$ then $I^{-1}\langle I\mcal{C}'\rangle =\mcal{C}'$.
\item If $\mcal{D}'$ is a definable subcategory of $\langle I\mcal{C}\rangle$ then $\langle I(I^{-1}\mcal{D}')\rangle=\mcal{D}'$.
\end{enumerate}

\end{lemma}
\begin{proof}
$(i)$ Suppose $M\in\mcal{C}'$. Then $IM\in\langle I\mcal{C}' \rangle$. So $M\in I^{-1}\langle I\mcal{C}'\rangle$.

Suppose $N\in\pinj(\mcal{C})$ and $N\in I^{-1}\langle I\mcal{C}'\rangle$. If $IN=0$ then $N\in \mcal{C}'$ since $\ker I\subseteq \mcal{C}'$. So we may assume that $IN\neq 0$ and $IN$ is a pure-subobject of $IL$ for some $L\in\mcal{C}'$ by \ref{imintfunctorpuresubmod}. Since $N$ is pure-injective, so is $IN$. Hence $IN$ is a direct summand of $IL$. By the hypotheses on $I$, $IN$ is indecomposable.  So by \cite[18.2.24]{PSL}, there exists $L'\in\pinj(\mcal{C}')$ such that $IN$ is a direct summand of $IL'$. By the hypothesis on $I$, $IL'$ is indecomposable and hence $IN\cong IL'$. By the other hypothesis on $I$, $L'\cong N$. Thus $N\in \mcal{C}'$ as required.

Since definable subcategories are determined by the indecomposable pure-injective modules they contain, $I^{-1}\langle I\mcal{C}'\rangle\subseteq \mcal{C}'$.

$(ii)$ Suppose $\mcal{D}'$ is a definable subcategory of $\langle I\mcal{C}\rangle$.  Since $\mcal{D}'$ is a definable subcategory, $\langle I(I^{-1}\mcal{D}')\rangle\subseteq\mcal{D}'$ if and only if $ I(I^{-1}\mcal{D}')\subseteq\mcal{D}'$. Take $M\in I^{-1}\mcal{D}'$. By definition, $IM\in \mcal{D}'$. So $ I(I^{-1}\mcal{D}')\subseteq\mcal{D}'$.

We now show that $\mcal{D}'\subseteq \langle I(I^{-1}\mcal{D}') \rangle$. Suppose $N\in\pinj(\mcal{D}')$. Since $\mcal{D}'\subseteq \langle I\mcal{C}\rangle$, by \ref{imintfunctorpuresubmod}, there exists $L\in\mcal{C}$ such that $N$ is pure-subobject of $IL$. Thus $N$ is a direct summand of $IL$. By \cite[18.2.24]{PSL}, we may assume $L$ is also indecomposable pure-injective. Thus $N\cong IL$. So $L\in I^{-1}\mcal{D}'$ and $N\cong IL\in I(I^{-1}\mcal{D}')$ as required.
\end{proof}

\begin{cor}
Let $I:\mcal{C}\rightarrow \mcal{D}$ be an interpretation functor such that for all $N\in \pinj(\mcal{C})$, $IN=0$ or $IN\in\pinj(\mcal{D})$ and if $N,M\in \pinj(\mcal{C})$, $IN,IM\neq 0$ and $IN\cong IM$ then $N\cong M$. The maps
\[\ker I\subseteq\mcal{C}'\subseteq \mcal{C}\mapsto \langle I\mcal{C}'\rangle\]
and
\[\mcal{D}'\subseteq \langle I\mcal{C}\rangle\mapsto I^{-1}\mcal{D}'\]
give a inclusion preserving bijective correspondence between definable subcategories in $\langle I\mcal{C}\rangle$ and definable subcategories of $\mcal{C}$ containing $\ker I$.
\end{cor}
\begin{proof}
We have shown that if $\mcal{C}'$ is a definable subcategory of $\mcal{C}$ containing $\ker I$ then $I^{-1}\langle I\mcal{C}'\rangle=\mcal{C}'$ and if $\mcal{D}'$ is a definable subcategory of $\langle I\mcal{C}'\rangle$ then $\langle I(I^{-1}\mcal{D}')\rangle=\mcal{D}'$.

That this correspondence is inclusion preserving follows directly from its definition.
\end{proof}

The following is very close to \cite[18.2.26]{PSL}, \cite[3.19]{Interpretingmodules} and \cite[7.8]{Exdefcat} but our hypotheses are slightly different. This statement will be needed in section \ref{S-duality}.

\begin{proposition}\label{superintZg}
Let $I:\mcal{C}\rightarrow \mcal{D}$ be an interpretation functor such that for all $N\in \pinj(\mcal{C})$, $IN=0$ or $IN\in\pinj(\mcal{D})$ and if $N,M\in \pinj(\mcal{C})$, $IN,IM\neq 0$ and $IN\cong IM$ then $N\cong M$. The assignment $N\mapsto IN$ induces a homeomorphism between $\Zg(\mcal{C})\backslash \ker I$ and its image in $\Zg(\mcal{D})$ which is closed.
\end{proposition}
\begin{proof}
Suppose $L\in \langle I\mcal{C}\rangle\cap\Zg(\mcal{D})$. Then $L$ is a pure-subobject of some $IN$ for some $N\in\Zg(\mcal{C})$. By hypothesis on $I$, $IN$ is indecomposable. So $L\cong IN$. Thus the closed set $\langle I\mcal{C}\rangle\cap\Zg(\mcal{D})$ is the image of $\Zg(\mcal{C})\backslash \ker I$ under $I$.

Suppose $X$ is a closed subset of $\Zg(\mcal{D})$ contained in $I\Zg(\mcal{C})$. Let $\mcal{X}$ be the definable subcategory of $\mcal{D}$ generated by $X$. Let $\mcal{Y}:=I^{-1}\mcal{X}$ and $Y:=\mcal{Y}\cap \Zg(\mcal{C})$. Since $\mcal{X}\subseteq \langle I\mcal{C}\rangle$, $IL\in \mcal{X}$ if and only if $L\in\mcal{Y}$ by \ref{defsubcatunderIinv}. So $N\in Y$ if and only if $IN\in X$. Thus $N\mapsto IN$ is continuous.

Suppose $Y$ is a closed subset of $\Zg(\mcal{C})$. We may replace $Y$ by the closed subset $Y\cup(\ker I\cap\Zg(\mcal{C}))$ without changing its intersection with $\Zg(\mcal{C})\backslash \ker I$. Let $\mcal{Y}$ be the definable subcategory of $\mcal{C}$ generated by $Y$ and let $X=\langle I\mcal{Y}\rangle\cap \Zg(\mcal{D})$. Now $N\in \mcal{Y}$ if and only $N\in I^{-1}\langle I\mcal{Y}\rangle$ by \ref{defsubcatunderIinv}. So $N\in Y$ if and only if $IY\in X$. Thus the inverse of $N\mapsto IN$ is continuous.
\end{proof}

\section{Maranda's functor}\label{SMarfun}

Throughout this section, $R$ will be a discrete valuation domain with field of fractions $Q$ and maximal ideal generated by $\pi$, and $\Lambda$ will be an $R$-order such that $Q\Lambda$ is a separable $Q$-algebra.

The basis of Maranda's theorem is the existence of a non-negative integer $l$ such that for all $\Lambda$-lattices $L$ and $M$,
\[\pi^{l}\Ext^1(L,M)=0.\] Throughout this section, let $k_0$ be the smallest such non-negative integer. We will call this natural number \textbf{Maranda's constant} (for $\Lambda$ as an $R$-order).

Note that since $\Lambda$ is noetherian, $\Ext^1(L,-)$ is finitely presented as a functor in $(\mod\text{-}\Lambda,\Ab)$ (see \cite[10.2.35]{PSL}). Hence $\pi^{k_0}\Ext^1(L,-)$ is also finitely presented. Since $\Tf_{\Lambda}$ is the smallest definable subcategory containing $\Latt_\Lambda$, $\pi^{k_0}\Ext^1(L,N)=0$ for all $L\in\Latt_\Lambda$ and $N\in\Tf_{\Lambda}$.

Throughout this section, when $k\in\N$ is clear from the context, for $M\in\Mod\text{-}\Lambda$ and $m\in M$, we will often write $\overline{M}$ for $M/M\pi^k$ and $\overline{m}$ for $m+M\pi^k$. If $f:M\rightarrow N\in \Mod\text{-}\Lambda$ then we will write $\overline{f}$ for the induced homomorphism from $M/M\pi^k$ to $N/N\pi^k$. This is to allow us to use subscripts on modules as indices and to ease readability.  We will write $\Lambda_k$ for the ring $\Lambda/\pi^k\Lambda$.

The proof of the following lemma can easily be extracted from the proof of \cite[30.14]{CurtisReinerV1}.

\begin{lemma}\label{marandalemmalatttotf}
Let $L\in\Latt_\Lambda$ and $M\in\Tf_\Lambda$. If $k\geq k_0+1$ then for all $g\in \Hom_{\Lambda_k}(L/L\pi^k,M/M\pi^k)$ there exists $h\in \Hom_\Lambda(L,M)$ such that for all $m\in L$, $\pi^{k-k_0}+\Lambda\pi^k|\overline{h(m)}-g(\overline{m})$.
\end{lemma}
%\begin{proof}
%The proof of Maranda's theorem in [CR, 30.14] shows that for all $g\in \Hom_{\Lambda_k}(L_k,M_k)$ there exists $f\in \Hom_\Lambda(L,M)$ such that for all $m\in L$, $\pi^{k-k_0}|\overline{f(m)}-g(\overline{m})$ whenever $L$ is a $\Lambda$-lattice and $M\in \Tf_\Lambda$.
%\end{proof}

The following proposition is key to proving both parts of our extension of Maranda's theorem.

\begin{proposition}\label{uremaranda}
Let $M,N$ be $R$-torsion-free $\Lambda$-modules with $N$ pure-injective. If $k\geq k_0+1$ then for all $g\in \Hom_{\Lambda_k}(M/M\pi^k,N/N\pi^k)$ there exists $h\in \Hom_\Lambda(M,N)$ such that for all $m\in M$, $\pi^{k-k_0}+\Lambda\pi^k|\overline{h(m)}-g(\overline{m})$.
\end{proposition}
\begin{proof}

Since $M\in \Tf_\Lambda$, there exists a directed system of $\Lambda$-lattices $L_i$ for $i\in I$ and $\sigma_{ij}:L_i\rightarrow L_j$ for $i\leq j\in I$ such that $M$ is the direct limit of this directed system. Let $f_i:L_i\rightarrow M$ be the component maps.

Our aim is to find $h_i:L_i\rightarrow N$ for all $i\in I$ such that $h_i=h_j\sigma_{ij}$ and for all $a\in L_i$, $\pi^{k-k_0}+\Lambda\pi^k|\overline{h_i(a)}-g(\overline{f_i(a)})$.

If we can do this then there exists $h:M\rightarrow N$ such that $h_i=hf_i$ for all $i\in I$. This homomorphism is then as required by the statement of the proposition for the following reasons. For all $m\in M$, there exist $i\in I$ and $a\in L_i$ such that $f_i(a)=m$. So \[\overline{h(m)}-g(\overline{m})=\overline{hf_i(a)}-g(\overline{f_i(a)})=\overline{h_i(a)}-g(\overline{f_i(a)})\] is divisible by $\pi^{k-k_0}+\Lambda\pi^k$.

For each $i\in I$, let $\epsilon_i:L_i\rightarrow N$ be such that for all $a\in L_i$, $\pi^{k-k_0}+\Lambda\pi^k$ divides $\overline{\epsilon_i(a)}-g(\overline{f_i(a)})$. Such an $\epsilon_i$ exists by \ref{marandalemmalatttotf} since $L_i$ is a $\Lambda$-lattice.

Let $\mathbf{c_i}:=(c_{i1},\ldots, c_{il_i})$ generate $L_i$ as an $R$-module and $\phi_i$ generate the pp-type of $\mathbf{c_i}$. Note that $\mathbf{m}\in\phi_i(N)$ if and only if there exists a $q:L_i\rightarrow N$ such that $q(\mathbf{c_i})=\mathbf{m}$.

Let \[\chi_i(x_1,\ldots, x_{l_i}):=\phi_i(x_1,\ldots,x_{l_i})\wedge\bigwedge_{j=1}^{l_i}\pi^{k-k_0}|x_j.\]

We now show that $\mathbf{m}-\epsilon_i(\mathbf{c_i})\in\chi_i(N)$ if and only if there exists a homomorphism $q\in \Hom(L_i,N)$ such that $q(\mathbf{c_i})=\mathbf{m}$ and for all $a\in L_i$, $\pi^{k-k_0}+\Lambda\pi^k$ divides $\overline{q(a)}-g(\overline{f_i(a)})$.

Suppose $\mathbf{m}-\epsilon_i(\mathbf{c}_i)\in\chi_i(N)$. Since $\epsilon_i(\mathbf{c_i})\in\phi_i(N)$, $\mathbf{m}\in\phi_i(N)$ and hence there exists $q\in\Hom(L_i,N)$ such that $q(\mathbf{c_i})=\mathbf{m}$. For each $1\leq j\leq l_i$, $\pi^{k-k_0}$ divides $q(c_{ij})-\epsilon(c_{ij})=m_j-\epsilon_i(c_{ij})$. By definition of $\epsilon_i$, $\pi^{k-k_0}+\Lambda\pi^k$ divides $\overline{\epsilon_i(c_{ij})}-g(\overline{f_i(c_{ij})})$. So $\pi^{k-k_0}+\Lambda\pi^k$ divides $q(c_{ij})-g(\overline{f_i(c_{ij})})$ for $1\leq j\leq l_i$. Since $\mathbf{c_i}$ generates $L_i$, $\pi^{k-k_0}+\Lambda\pi^k$ divides $\overline{q(a)}-g(\overline{f_i(a)})$ for all $a\in L_i$.

Now suppose that $q\in\Hom(L_i,N)$ is such that $q(\mathbf{c_i})=\mathbf{m}$ and that for all $a\in L_i$, $\pi^{k-k_0}+\Lambda\pi^k$ divides $\overline{q(a)}-g(\overline{f_i(a)})$. Then $\mathbf{m}-\epsilon_i(\mathbf{c_i})=(q-\epsilon_i)(\mathbf{c_i})\in\phi_i(N)$. By definition of $\epsilon_i$, $\pi^{k-k_0}+\Lambda\pi^k$ divides $\overline{\epsilon_i(a)}-g(\overline{f_i(a)})$ for all $a\in L_i$. So $\pi^{k-k_0}+\Lambda\pi^k$ divides $\overline{q(a)}-\overline{\epsilon_i(a)}$ for all $a\in L_i$. Since $k\geq k-k_0$, $\pi^{k-k_0}$ divides $q(a)-\epsilon_i(a)$ for all $a\in L_i$. So, in particular, $\pi^{k-k_0}$ divides $q(c_{ij})-\epsilon_i(c_{ij})=m_j-\epsilon_i(c_{ij})$ for all $1\leq j\leq l_i$. Thus $\mathbf{m}-\epsilon_i(c_i)\in\chi_i(N)$ as required.

For $i\leq j\in I$, let $\mathbf{t_{ij}}\in R^{l_j\times l_i}$ be such that $\sigma_{ij}(\mathbf{c_i})=\mathbf{c_j}\cdot\mathbf{t_{ij}}$.

Consider the system of linear equations and cosets of pp-definable subsets
\[\mathbf{x_i}\in \epsilon_i(\mathbf{c_i})+\chi_{i}(N)\tag*{$(1)_i$}\] for $i\in I$ and
\[\mathbf{x_i}=\mathbf{x_j}\cdot\mathbf{t_{ij}}\tag*{$(2)_{ij}$}\] for $i\leq j\in I$.

Let $I_0\subseteq I$ be a finite subset of $I$. Since $I$ is directed, by adding an element to $I_0$ if necessary, we may assume that there is a $p\in I_0$ such that $i\leq p$ for all $i\in I_0$.

Let $\mathbf{m_p}=\epsilon_p(\mathbf{c_p})$ and for $i\in I_0$, let $\mathbf{m_i}=\mathbf{m_p}\cdot \mathbf{t_{ip}}$.  Then \[\mathbf{m_i}=\epsilon_p(\mathbf{c_p})\cdot \mathbf{t_{ip}}=\epsilon_p(\mathbf{c_p}\cdot \mathbf{t_{ip}})=\epsilon_p(\sigma_{ip}(\mathbf{c_i}))\] for all $i\in I_0$.

Suppose that $i\leq j\in I_0$. Then $\sigma_{ip}=\sigma_{jp}\circ\sigma_{ij}$. So \[\mathbf{m_i}=\epsilon_p(\sigma_{jp}\circ\sigma_{ij}(\mathbf{c_i}))=\epsilon_p(\sigma_{jp}(\mathbf{c_j}\cdot \mathbf{t_{ij}}))=\epsilon_p(\sigma_{jp}(\mathbf{c_j}))\cdot\mathbf{t_{ij}}=\mathbf{m_j}\cdot\mathbf{t_{ij}}.\] Thus $(\mathbf{m_i})_{i\in I_0}$ satisfies $(2)_{ij}$ for all $i\leq j\in I_0$.

We now need to show that for all $i\in I_0$, $\mathbf{m_i}-\epsilon_i(\mathbf{c_i})\in \chi_i(N)$. Since $\mathbf{m_i}=\epsilon_p\circ\sigma_{ip}(\mathbf{c_i})$, $\mathbf{m_i}\in\phi_i(N)$. So, since $\epsilon_i(\mathbf{c_i})\in\phi_i(N)$, $\mathbf{m_i}-\epsilon_i(\mathbf{c_i})\in\phi_i(N)$.

%For all $a\in L_i$, \[\overline{\epsilon_p\circ\sigma_{ip}(a)}-g(\overline{f_i(a)})=\overline{\epsilon_p(\sigma_{ip}(a))}-g(\overline{f_p(\sigma_{ip}(a))}).\]

By definition of $\epsilon_p$ and $\epsilon_i$, $\pi^{k-k_0}+\Lambda\pi^k$ divides $\overline{\epsilon_p(\sigma_{ip}(a))}-g(\overline{f_p(\sigma_{ip}(a))})$ and $\pi^{k-k_0}+\Lambda\pi^k$ divides $\overline{\epsilon_i(a)}-g(\overline{f_i(a)})$ for all $a\in L_i$. So, since $f_p(\sigma_{ip}(a))=f_i(a)$, $\pi^{k-k_0}+\Lambda\pi^k$ divides
$\overline{\epsilon_p(\sigma_{ip}(a))}-\overline{\epsilon_i(a)}$ for all $a\in L_i$. Thus, using the characterisation of $\chi_i$ proved earlier, $\mathbf{m_i}-\epsilon_i(\mathbf{c}_i)=\epsilon_p\circ\sigma_{ip}(\mathbf{c_i})-\epsilon_i(\mathbf{c}_i) \in \chi_i(N)$.

%So for $1\leq j\leq l_i$, $\pi^{k-k_0}$ divides $\overline{\epsilon_p(\sigma_{ip}(c_{ij}))-\epsilon_i(c_{ij})}$. Since $k>k-k_0$, this implies that for $1\leq j\leq l_i$, $\pi^{k-k_0}$ divides $\epsilon_p(\sigma_{ip}(c_{ij}))-\epsilon_i(c_{ij})=m_j-\epsilon_i(c_{ij})$. Therefore $\mathbf{m_i}-\epsilon_i(\mathbf{c_i})\in \chi_i(N)$. Thus $(\mathbf{m_i})_{i\in I_0}$ satisfies $(1)_{i}$ for all $i\in I_0$.

Since the system of equations $(1)_i, (2)_{ij}$ is finitely solvable and $N$ is pure-injective, there exists $(\mathbf{m_i})_{i\in I}$ with $\mathbf{m_i}\in N$ satisfying $(1)_{i}$ and $(2)_{ij}$ for all $i\leq j\in I$. For each $i\in I$, let $h_i:L_i\rightarrow N$ be the homomorphism which sends $\mathbf{c_i}$ to $\mathbf{m_i}$. Condition $(2)_{ij}$ ensures that for all $i\leq j\in I$, $h_i=h_j\circ\sigma_{ij}$. This is because $h_j(\sigma_{ij}(\mathbf{c_i}))=h_j(\mathbf{c_j}\cdot \mathbf{t_{ij}})=h_j(\mathbf{c_j})\cdot \mathbf{t_{ij}}=\mathbf{m_j}\cdot \mathbf{t_{ij}}=\mathbf{m_i}$. Condition $(1)_i$ ensures that $\pi^{k-k_0}+\Lambda\pi^k$ divides $\overline{h_i(a)}-g(\overline{f_i(a)})$ for all $a\in L_i$. \end{proof}

\begin{lemma}\label{isopluspidivisiblemapisiso}
Let $N\in \Mod\text{-}\Lambda_k$ and $g,\sigma\in \End N$. Suppose that for all $m\in N$, $\pi+\Lambda\pi^k|\sigma(m)$. Then $g-\sigma$ is an isomorphism if and only if $g$ is an isomorphism.
\end{lemma}
\begin{proof}
Suppose that $g$ is an isomorphism. Then $(g-\sigma)g^{-1}=\text{Id}_N-\sigma g^{-1}$. Let $h:=\sigma g^{-1}$ and $f:=\text{Id}_N+h+\ldots h^{k-1}$. Since $\pi+\Lambda\pi^k|\sigma(m)$ for all $m\in N$, $h^k=0$. Thus $(\text{Id}_N-h)\circ f=f\circ(\text{Id}_N-h)=\text{Id}_N$. So $(g-\sigma)g^{-1}f=\text{Id}_N$ and $g^{-1}f(g-\sigma)=g^{-1}f(g-\sigma)g^{-1}g=\text{Id}_N$ Therefore $g-\sigma$ is an isomorphism.
\end{proof}

\begin{theorem}\label{marandapresisotype}
Let $M,N\in\Tf_\Lambda$ be $R$-reduced and pure-injective. If $M/M\pi^k\cong N/N\pi^k$ for some $k\geq k_0+1$ then $M\cong N$.
\end{theorem}
\begin{proof}
We first show that if $f:M\rightarrow N$ is such that $\overline{f}:\overline{M}\rightarrow \overline{N}$ is an isomorphism, then $f$ is an isomorphism.

Suppose $\overline{f}$ is an isomorphism and $f(m)=0$. If $m\neq 0$ then there exists $n\in M$ such that $m=n\pi^l$ and $\pi$ does not divide $n$ since $M$ is reduced. Since $N$ is $R$-torsion-free $f(m)=f(n)\pi^l=0$ implies $f(n)=0$. So $\overline{f}(\overline{n})=0$. Therefore $\overline{n}=0$. This implies $\pi$ divides $n$ contradicting our assumption. So $m=0$. Therefore $f$ is injective.

We now show that $f$ is surjective. Since $\overline{f}$ is surjective, for all $n\in N$ there exists $m\in M$ such that $n-f(m)\in N\pi^k$. Suppose $m_l$ is such that $n-f(m_l)\in N\pi^{lk}$. Let $a\pi^{lk}=n-f(m_l)$. There exists $b\in M$ such that $a-f(b)\in N\pi^{k}$. Thus $a\pi^{lk}-f(b)\pi^{lk}\in N\pi^{(l+1)k}$. So $n-f(b\pi^{lk}+m_l) \in N\pi^{(l+1)k}$ and $(b\pi^{lk}+m_l)-m_l\in N\pi^{lk}$. So there exists a sequence $(m_l)_{l\in\N}$ in $M$ such that for all $l\in \N$, $n-f(m_l)\in N\pi^{lk}$ and $m_{l+1}-m_l\in M\pi^{lk}$. Since $M$ is pure-injective, there exist an $m\in M$ such that $m-m_l\in M\pi^{kl}$ for all $l\in \N$. Thus $f(m)-n=f(m-m_l)-(n-f(m_l))\in N\pi^{kl}$ for all $l\in\N$. Since $N$ is reduced, $f(m)=n$.

Suppose that $g:\overline{M}\rightarrow \overline{N}$ is an isomorphism with inverse $h:\overline{N}\rightarrow \overline{M}$. There exists $e\in \Hom_{\Lambda}(M,N)$ such that for all $m\in M$, $\pi^{k-k_0}+\Lambda\pi^{k}$ divides $\overline{e(m)}-g(\overline{m})$ and $f\in \Hom_{\Lambda}(N,M)$ such that for all $m\in N$, $\pi^{k-k_0}+\Lambda\pi^k$ divides $\overline{f(m)}-h(\overline{m})$. Since $\overline{f}\circ \overline{e}= (\overline{f}-h)\circ (\overline{e}-g)+(\overline{f}-h)\circ g+ h\circ(\overline{e}-g)+h\circ g$, \ref{isopluspidivisiblemapisiso} implies that $\overline{f}\circ\overline{e}$ is an isomorphism. Similarly, we can show that $\overline{e}\circ\overline{f}$ is an isomorphism. Thus $\overline{e}$ and $\overline{f}$ are both isomorphisms. So the above arguments imply that $e$ and $f$ are both isomorphisms.
\end{proof}

\begin{theorem}\label{marandapresind}
Let $k\geq k_0+1$. If $N$ is an indecomposable $R$-torsion-free $R$-reduced pure-injective $\Lambda$-module then $N/N\pi^k$ is indecomposable.
\end{theorem}
\begin{proof}
We will show that for all $f\in \End \overline{N}$, either $f$ is an isomorphism or $1-f$ is an isomorphism. Hence $\End \overline{N}$ is local.

Proposition \ref{uremaranda} implies that the homomorphism sending $f\in \End N$ to $\overline{f}\in\End\overline{N}$ induces a surjective ring homomorphism from $\End N$ to $\End \overline{N}/\{g\in\End \overline{N} \st g(n)\in \overline{N}\pi  \text{ for all } n\in \overline{N} \}$.

Suppose $f\in \End \overline{N}$ is not an isomorphism. There exists $g\in \End N$ and $\sigma\in\End \overline{N}$ such that $f=\overline{g}+\sigma$ and $\sigma(n)\in\overline{N}\pi$ for all $n\in\overline{N}$. By \ref{isopluspidivisiblemapisiso}, $\overline{g}$ is not an isomorphism and hence neither is $g$. Since $\End N$ is local, $\text{Id}_N-g$ is an isomorphism. Thus $\text{Id}_{\overline{N}}-\overline{g}$ is an isomorphism. So by \ref{isopluspidivisiblemapisiso}, $\text{Id}_{\overline{N}}-f=\text{Id}_{\overline{N}}-(\overline{g}+\sigma)$ is an isomorphism, as required.
%
%
%We will show that the endomorphism ring of $N/N\pi^k$ is local. Proposition \ref{uremaranda} implies that the homomorphism sending $f\in \End N$ to $\overline{f}\in\End\overline{N}$ induces a surjective ring homomorphism from $\End N$ to $\End \overline{N}/\{g\in\End \overline{N} \st g(n)\in \overline{N}\pi  \text{ for all } n\in \overline{N} \}$. Thus, in order to show that $\End \overline{N}$ is local, it is enough to show that all maximal ideals of $\End \overline{N}$ contain $\{g\in\End \overline{N} \st g(n)\in \overline{N}\pi \text{ for all } n\in \overline{N} \}$.
%
%Suppose $h\in \End \overline{N}$ is not invertible and $g(n)\in \overline{N}\pi$ for all $n\in \overline{N}$.  Then $h+g$ is not invertible by \ref{isopluspidivisiblemapisiso}. Thus every maximal ideal contains $\{g\in\End \overline{N} \st g(n)\in \overline{N}\pi \text{ for all } n\in \overline{N} \}$. Thus $\End \overline{N}$ is local.
\end{proof}

We now show that Maranda's functor preserves pure-injective hulls. The proof uses somewhat different techniques to those used so far and relies on \cite[4.6]{TfpartZgorders}. In order to avoid introducing various definitions that will not be used in the rest of this paper, we state only the part of that proposition which we need.

\begin{proposition}\label{clasMarlattiso}
Let $k\geq k_0+1$. For all $\psi\in [\pi^{k-k_0}|\mathbf{x},\mathbf{x}=\mathbf{x}]\subseteq \pp^n_{\Lambda}$ there exists $\widehat{\psi}\in[\pi^{k-k_0}+\Lambda\pi^k|\mathbf{x},\mathbf{x}=\mathbf{x}]\subseteq \pp_{\Lambda_k}^n$ such that for all $M\in\Tf_\Lambda$ and $\mathbf{m}\in M$, $\mathbf{m}\in\psi(M)$ if and only if $\mathbf{m}+M\pi^k\in\widehat{\psi}(M/M\pi^k)$.
\end{proposition}

The following useful lemma was communicated to me by Mike Prest.

\begin{lemma}\label{onevariable}
Let $M\in\Mod\text{-}S$, $H(M)$ be its pure-injective hull and $\mathbf{b}\in H(M)$ be an $n$-tuple. Suppose that $\mathbf{b}\in \phi(H(M))\backslash\bigcup_{i=1}^l\psi_i(H(M))$ where $\phi,\psi_1,\ldots,\psi_n$ are pp-$n$-formulas. There exists an $n$-tuple $\mathbf{b}'\in M$ and a pp-$n$-formula $\theta$ such that $\theta(\mathbf{b}'-\mathbf{b})$ holds and
\[H(M)\models \theta(\mathbf{b}'-\mathbf{y})\rightarrow \phi(\mathbf{y})\wedge\bigwedge_{i=1}^n\neg\psi_i(\mathbf{y}).\]
\end{lemma}
\begin{proof}
Let $\mathbf{b}\in H(M)$. Suppose that $\mathbf{b}\in\phi(H(M))$ and $\mathbf{b}\notin\bigcup_{i=1}^l\psi_i(H(M))$.

By \cite[4.10]{MikesBluebook}, there exists $\mathbf{a}\in M$ and a pp formula $\chi(\mathbf{x},\mathbf{y})$ such that $\chi(\mathbf{a},\mathbf{b})$ holds in $H(M)$ and

\[H(M)\models\chi(\mathbf{a},\mathbf{y})\rightarrow \phi(\mathbf{y})\wedge\bigwedge_{i=1}^n\neg\psi_i(\mathbf{y}).\]

Since $H(M)$ is an elementary extension of $M$, there exists $\mathbf{b}'\in M$ such that $\chi(\mathbf{a},\mathbf{b}')$ holds in $M$ and hence in $H(M)$. Thus $\chi(\mathbf{0},\mathbf{b}'-\mathbf{b})$ holds in $H(M)$. Set $\theta(\mathbf{z}):=\chi(\mathbf{0},\mathbf{z})$. So $\theta(\mathbf{b}'-\mathbf{b})$ holds in $H(M)$.

Suppose $\mathbf{c}\in H(M)$ and $\theta(\mathbf{b}'-\mathbf{c})$ holds in $H(M)$. Then $\chi(\mathbf{a},\mathbf{c})$ holds in $H(M)$. Thus $\phi(\mathbf{c})\wedge\bigwedge_{i=1}^l\neg\psi_i(\mathbf{c})$ holds in $H(M)$. So $\theta(\mathbf{b}'-\mathbf{b})$ holds and

\[H(M)\models \theta(\mathbf{b}'-\mathbf{y})\rightarrow \phi(\mathbf{y})\wedge\bigwedge_{i=1}^l\neg\psi_i(\mathbf{y}).\]
\end{proof}

The following theorem is motivated by \cite[3.16]{Interpretingmodules}.

%The following theorem is motivated by \cite[3.16]{Interpretingmodules} which says that if an interpretation functor $F$, with underlying pp-pair $\phi/\psi$, preserves the full induced structure on $\phi/\psi$ then $F$ preserves pure-injective hulls. Let $\mcal{C}\subseteq \Mod\text{-}S$, $\mcal{D}\subseteq \Mod\text{-}T$ and let $F:\mcal{C}\rightarrow \mcal{D}$ be an interpretation functor with underlying pp-pair $\phi/\psi$. The functor $F$ is said to \textbf{preserve the full induced structure} on its underlying pp-pair $\phi/\psi$ if for all $\sigma\in [\psi(\mathbf{x_1})\wedge\cdots\wedge\psi(\mathbf{x_l}),\phi(\mathbf{x_1})\wedge\cdots\wedge\phi(\mathbf{x_l})]$ there exists $\sigma'\in\pp^l_T$ such that for all $M\in\mcal{C}$ and $(\mathbf{m_1},\ldots,\mathbf{m_l})\in \phi(M)^l$, $(\mathbf{m_1},\ldots,\mathbf{m_l})\in\sigma(M)$ if and only if $(\mathbf{m_1}+\psi(M),\ldots,\mathbf{m_l}+\psi(M))\in\sigma'(FM)$.

\begin{theorem}\label{prespihulls}
Let $k\geq k_0+1$ and $M\in\Tf_\Lambda$. If $u:M\rightarrow H(M)$ is a pure-injective hull of $M$ then the induced map $\overline{u}:M/M\pi^k\rightarrow H(M)/H(M)\pi^k$ is a pure-injective hull for $M/M\pi^k$.
\end{theorem}
\begin{proof}
We identify $M$ with its image in $H(M)$. Our aim is to show that for all $b\in H(M)$ there exists $a\in M$ and $\chi(x,y)\in\pp^2_{\Lambda_k}$ such that $\chi(\overline{a},\overline{b})$ holds in $H(M)/H(M)\pi^k$ and $\chi(\overline{a},\overline{0})$ does not hold in $H(M)/H(M)\pi^k$.

Suppose that $\pi$ does not divide $b\in H(M)$. Since $H(M)$ is the pure-injective hull of $M$, by \ref{onevariable}, there exists $a\in M$ and a pp formula $\theta(x)\in\pp^1_\Lambda$ such that $\theta(a-b)$ holds in $H(M)$ and $\theta(a-x)\rightarrow \neg\pi|x$. Let $\Delta(x):=\theta(x)+\pi|x$. Then $\Delta(a-b)$ holds in $H(M)$ and for all $c\in H(M)$, $\Delta(a-c\pi)$ does not hold. Let $\widehat{\Delta}$ be as in \ref{clasMarlattiso}. So $\widehat{\Delta}(\overline{a}-\overline{b})$ holds in $H(M)/H(M)\pi^k$ and $\widehat{\Delta}(\overline{b})$ does not hold in $H(M)/H(M)\pi^k$.

Now suppose that $e\in H(M)\backslash H(M)\pi^k$, $e=b\pi^n$ and $\pi$ does not divide $b$. Note that this implies $n<k$. Let $\Delta$ and $a\in M$ be as in the previous paragraph i.e. $\Delta\geq \pi|x$, $\Delta(a-b)$ holds in $H(M)$ and for all $c\in H(M)$, $\Delta(a-c\pi)$ does not hold. Let $\chi(x,y):=\exists z \ \widehat{\Delta}(x-z)\wedge y=z\pi^n\in \pp^2_{\Lambda_k}$. Suppose that $\chi(\overline{a},\overline{0})$ holds. Then there exists $d\in H(M)$ such that $\overline{d}\pi^n=\overline{0}$ and $\widehat{\Delta}(\overline{a}-\overline{d})$ holds. But then $d\pi^n\in H(M)\pi^k$. Since $M$ and hence $H(M)$ is $R$-torsion-free, $d\in H(M)\pi^{k-n}$. This contradicts the definition of $\Delta$. Thus $\chi(\overline{a},\overline{e})$ holds and $\chi(\overline{a},\overline{0})$ does not hold in $H(M)/H(M)\pi^k$.

Suppose that $H(M)/H(M)\pi^k=N\oplus N'$ and $M/M\pi^k\subseteq N$. If $\overline{c}\in H(M)/H(M)\pi^k$ is non-zero then we have shown that there exist $\overline{a}\in \overline{M}$ and $\chi(x,y)\in\pp_{\Lambda_k}^2$ such that $\chi(\overline{a},\overline{c})$ holds and $\chi(\overline{a},\overline{0})$ does not hold. Since the solution sets of pp formulas commute with direct sums, this implies that if $\overline{c}\in N'$ then $\overline{c}=\overline{0}$. Thus $N'$ is the zero module and $H(M)/H(M)\pi^k$ is the pure-injective hull of $M/M\pi^k$.
\end{proof}

\section{Pure-injectives and pure-injective hulls}\label{S-Piandpihulls}
As in the previous section, $R$ will be a discrete valuation domain with field of fractions $Q$ and maximal ideal generated by $\pi$, and $\Lambda$ will be an $R$-order such that $Q\Lambda$ is a separable $Q$-algebra.

We start this section by showing that the pure-injective hull of an $R$-reduced $R$-torsion-free $\Lambda$-module is $R$-reduced. The proof of the following remark is the same as \cite[Claim 2, p. 1128]{TFpartZgRG}.

\begin{remark}\label{divimpliesinj}
If $M\in\Tf_\Lambda$ is $R$-divisible then $M$ is injective as a $\Lambda$-module.
%Let $R$ be a discrete valuation domain, $Q$ the field of fractions of $R$ and $\Lambda$ an $R$-order such that $Q\Lambda$ is separable (semisimple is enough).
\end{remark}

This allows us to deduce that all $M\in\Tf_\Lambda$ decompose as the direct sum of the divisible part $D_M$ of $M$ and an $R$-reduced module. Explicitly, let
\[D_M:=\{m\in M \st \pi^n|m \text{ for all } n\in\N\}.\] It is easy to check that $D_M$ is $R$-divisible. So, since $R$-divisible $R$-torsionfree $\Lambda$-modules are injective, $D_M$ is a direct summand of $M$. Hence $M\cong D_M\oplus M/D_M$. Now note that if $m\in M$ and $\pi^n|m+D_M$ for all $n\in\N$ then $\pi^n|m$ for all $n\in\N$. Thus $M/D_M$ is $R$-reduced.

\begin{lemma}\label{shiftingcomplementsofinjectives}
Let $S$ be a ring, $C,M,E\in \Mod\text{-}S$ and $E$ injective. Suppose that $C,E\subseteq M$ and $C\cap E=\{0\}$. There exist $N'\subseteq M$ such that $C\subseteq N'$ and $N'\oplus E=M$.
\end{lemma}
\begin{proof}
Using injectivity of $E$, there is an $f:M\rightarrow E$ such that $f|_C=0$ and $f|_E=\text{Id}_E$. So $C\subseteq \ker f$ and $M=E\oplus \ker f$.
%Let $i:C\oplus E\rightarrow M$ be the canonical inclusion and let $\lambda:C\oplus E\rightarrow E$ be the canonical projection. Since $E$ is injective, there exists $f:M\rightarrow E$ such that $f\circ i=\lambda$. Let $\mu:E\rightarrow C\oplus E$ be the canonical inclusion. Then $(f\circ i)\circ\mu=\lambda\circ \mu=1_{E}$. So $f\circ(i\circ \mu)= 1_{E}$. Thus $e=(i\circ\mu)\circ f$ is an idempotent endomorphism of $M$. For all $c\in C$, $e(c)=0$ and for all $m\in E$, $e(m)=m$. So $C\subseteq \ker e$ and $E\subseteq \ker e-1_M$. Since $E$ is injective there exists $T\subseteq \ker e-1_M$ such that $T\oplus E=\ker e-1_M$. So $(\ker e\oplus T)\oplus E=M$ and $C\subseteq (\ker e\oplus T)$.
\end{proof}

\begin{lemma}\label{redimpliesredpihull}
If $C\in\text{Tf}_\Lambda$ is $R$-reduced then $H(C)$ is $R$-reduced.
\end{lemma}
\begin{proof}
Since $Q\Lambda$ is separable, $H(C)=N\oplus D_{H(C)}$. Since $C$ is pure in $H(C)$ and $C$ is reduced, $C\cap D_{H(C)}=\{0\}$. By \ref{shiftingcomplementsofinjectives}, there exists $N'\subseteq H(C)$ such that $N'\oplus D_{H(C)}=H(C)$ and $C\subseteq N'$. Since $N$ and $N'$ are isomorphic, $N'$ is reduced. Since $N'$ is a direct summand of $H(C)$ and $C\subseteq N'\subseteq H(C)$, $N'=H(C)$. Thus $H(C)$ is $R$-reduced.
\end{proof}

\begin{definition}
If $M$ is a $\Lambda$-module then let $M^*$ denote the inverse limit along the canonical maps $M/M\pi^{n+1}\rightarrow M/M\pi^n$.
\end{definition}

\begin{remark}\label{redpicompletion}
Let $R$ be a discrete valuation domain and $M$ an $R$-reduced pure-injective $R$-module. Then $M$ is isomorphic to $M^*$.
\end{remark}
%\begin{proof}
%Let $M^*$ be the inverse limit and $f:M\rightarrow M^*$ be the homomorphism induced by the canonical maps from $M$ to $M/M\pi^n$. Since $M$ is reduced, $f$ is an embedding. Since $M$ is pure-injective (equivalently algebraically compact), $f$ is surjective.
%\end{proof}

\begin{theorem}
Let $M\in\Tf_\Lambda$. Then $M$ is pure-injective if and only if
\begin{enumerate}
\item $M/M\pi^k$ is pure-injective for all $k\in\N$ and
\item $M$ is pure-injective as an $R$-module.
\end{enumerate}
\end{theorem}
\begin{proof}
Certainly, if $M$ is pure-injective then conditions $(1)$ and $(2)$ hold.

So suppose that $(1)$ and $(2)$ hold. We know that $M$ is isomorphic to $D_M\oplus N$ and that $D_M$ is injective. Thus $M$ is pure-injective if and only if $N$ is pure-injective. Moreover, if conditions $(1)$ and $(2)$ hold for $M$ then they also hold of $N$. Let $H(N)$ be the pure-injective hull of $N$. Since $N/N\pi^k$ is pure-injective, \ref{prespihulls} implies that $H(N)/H(N)\pi^k=N/N\pi^k$. By \ref{redimpliesredpihull}, $H(N)$ is reduced and hence is isomorphic to $H(N)^*\cong N^*$. Since $N$ is reduced and pure-injective as an $R$-module, $N\cong N^*$. Thus $N\cong H(N)$ and is hence pure-injective. Thus $M=D_M\oplus N$ is also pure-injective.
\end{proof}

\begin{theorem}\label{pihullquotientspi}
Let $M\in \Tf_\Lambda$ be $R$-reduced and suppose that $M/M\pi^n$ is pure-injective for all $n\in\N$. Then the canonical map from $v:M\rightarrow M^*$ is the pure-injective hull of $M$.
\end{theorem}
\begin{proof}
Let $u:M\rightarrow H(M)$ be a pure-injective hull of $M$. For each $k\in\N$, let $u_k:M/M\pi^k\rightarrow H(M)/H(M)\pi^k$ be the homomorphism induced by $u$. For each $k\geq k_0+1$, $u_k:M/M\pi^k\rightarrow H(M)/H(M)\pi^k$ is the pure-injective hull of $M/M\pi^k$. Since $M/M\pi^k$ is pure-injective, $u_k$ is an isomorphism. The maps $u_k$ induce an isomorphism $w:M^*\rightarrow H(M)^*$. Since $M$ and hence, by \ref{redimpliesredpihull}, $H(M)$ is reduced, $H(M)\cong H(M)^*$. Viewing $H(M)^*$ as a submodule of $\prod_{i\in\N}H(M)/H(M)\pi^i$, for all $m\in M$, $wv(m)=(u(m)+H(M)\pi^i)_{i\in\N}$. Thus $v=w^{-1}u$.
\end{proof}

The same argument as used in the proof above shows that for any $R$-reduced $M\in\Tf_\Lambda$, the pure-injective hull of $M$ is $\varprojlim H(M/M\pi^i)$ along some surjective homomorphisms $p_i:H(M/M\pi^{i+1})\rightarrow H(M/M\pi^i)$. Unfortunately, it is not clear how to explicitly describe the homomorphisms $p_i$ beyond saying that $\ker p_i=H(M/M\pi^{i+1})\pi^{i}$.

For the rest of this section we focus on an application of \ref{pihullquotientspi}. We will calculate the pure-injective hull of the direct limit at the ``top'' of a generalised tube in $\Latt_\Lambda$. This will allow us to describe certain points of $\Zg_\Lambda^{tf}$ as modules when $\Lambda=\widehat{\Z}_{(2)}C_2\times C_2$ and answer the questions at the end of \cite{TFpartkleinfour}.

Following Krause in \cite{KrauseGeneric},  we define a \textbf{generalised tube} in $\mod\text{-}S$ to be a sequence of tuples $\mcal{T}:=(M_i,f_i,g_i)_{i\in\N_0}$ where $M_i\in \mod\text{-}S$, $M_0=0$, $f_i:M_{i+1}\rightarrow M_i$ and $g_i:M_i\rightarrow M_{i+1}$ such that for every $i\in\N$
\[\xymatrix@=20pt{%\xymatrix@R=20pt@C=40pt{
  M_i \ar[d]_{g_i} \ar[r]^{f_{i-1}} & M_{i-1} \ar[d]^{g_{i-1}} \\
  M_{i+1} \ar[r]^{f_{i}} & M_i   } \]
is a pushout and a pullback.

We will show that if $\mcal{T}$ is a generalised tube in $\Latt_\Lambda$ then its image, denoted $\mcal{T}_k$, in $\mod\text{-}\Lambda_k$ is a generalised tube.

Recall that a diagram
\[\xymatrix@=20pt{%\xymatrix@R=20pt@C=40pt{
  B \ar[d]_{a} \ar[r]^{b} & L \ar[d]^{g} \\
  M \ar[r]^{f} & P   } \]

is a pushout and a pullback if and only if

\[\xymatrix@C=0.5cm{
  0 \ar[r] & B \ar[rr]^{ \left(\begin{smallmatrix} a\\ b \end{smallmatrix}
                        \right)
  } && M\oplus L \ar[rr]^{ \left(\begin{smallmatrix} f&-g \end{smallmatrix}\right)
  } && P \ar[r] & 0 }\] is an exact sequence.

The following remark seems like it should be false because certainly Maranda's functor does not send monomorphisms between lattices to monomorphisms. Consider the exact sequence below. Since $M$ is projective as an $R$-module and $\beta$ is surjective, there exists $\gamma\in \Hom_R(M,N)$ such that $\beta\gamma=\text{Id}_M$. Thus the exact sequence is split when viewed as an exact sequence of $R$-modules. Therefore the second sequence is a split exact sequence of $R_k$-modules. Hence it is an exact sequence of $\Lambda_k$-modules.

\begin{remark}
If \[\xymatrix@C=0.5cm{
  0 \ar[r] & L \ar[rr]^{\alpha} && N \ar[rr]^{\beta} && M \ar[r] & 0 }\] is an exact sequence of $\Lambda$-lattices then
\[\xymatrix@C=0.5cm{
  0 \ar[r] & L_k \ar[rr]^{\overline{\alpha}} && N_k \ar[rr]^{\overline{\beta}} && M_k \ar[r] & 0 }\] is an exact sequences of $\Lambda_k$-modules.
\end{remark}

It follows that if $\mcal{T}$ is a generalised tube of $\Lambda$-lattices then $\mcal{T}_k:=((M_i)_k,\overline{f_i},\overline{g_i})_{i\in\N_0}$ is a generalised tube of finitely presented $\Lambda_k$-modules.

Given a generalised tube $\mcal{T}=(M_i,f_i,g_i)_{i\in\N_0}$, define $\mcal{T}[\infty]$ to be the direct limit along the embeddings $g_i:M_i\rightarrow M_{i+1}$.

%We now show that if $\mcal{T}=(M_i,f_i,g_i)_{i\in\N_0}$ is a generalised tube in $\Latt_\Lambda$ then $\mcal{T}[\infty]$ is reduced and $R$-torsion free but not pure-injective.
%
%As a direct limit of lattices, $\mcal{T}[\infty]$ is $R$-torsion free. Since each $f_i$ is spilt when viewed as a homomorphism of $R$-modules, $\mcal{T}[\infty]$ is isomorphic to $R^{(\aleph_0)}$ as an $R$-module. Since $R$ is not $\Sigma$-pure-injective as a module over itself, \cite[4.4.8]{PSL}, $R^{(\aleph_0)}$ is not pure-injective as an $R$-module and hence $\mcal{T}[\infty]$ is not pure-injective as an $\Lambda$-module. Finally, $R^{(\aleph_0)}$ is clearly reduced as an $R$-module, so $\mcal{T}[\infty]$ is reduced.

Recall that a module $M\in \Mod\text{-}S$ is \textbf{$\Sigma$-pure-injective} if $M^{(\kappa)}$ is pure-injective for every cardinal $\kappa$. Equivalently, \cite[4.4.5]{PSL}, $M$ is $\Sigma$-pure-injective if and only if $\pp^1_SM$ has the descending chain condition.

\begin{proposition}\label{gentubecatlatt}
Let $\mcal{T}=(M_i,f_i,g_i)_{i\in\N_0}$ be a generalised tube in $\Latt_\Lambda$. Then
\begin{enumerate}[(i)]
\item $\mcal{T}[\infty]$ is $R$-torsion-free and $R$-reduced,
\item $\mcal{T}[\infty]$ is not pure-injective,
\item for all $k\in\N$, $\mcal{T}[\infty]/\mcal{T}[\infty]\pi^k$ is $\Sigma$-pure-injective,
\item $\mcal{T}[\infty]^*$ is the pure-injective hull of $\mcal{T}[\infty]$.
\end{enumerate}
\end{proposition}
\begin{proof}
(i) \& (ii): As a direct limit of lattices, $\mcal{T}[\infty]$ is $R$-torsion free. Since each $f_i$ is spilt when viewed as a homomorphism of $R$-modules, $\mcal{T}[\infty]$ is isomorphic to $R^{(\aleph_0)}$ as an $R$-module. So $\mcal{T}[\infty]$ is reduced. Since $R$ is not $\Sigma$-pure-injective as a module over itself, \cite[4.4.8]{PSL}, $R^{(\aleph_0)}$ is not pure-injective as an $R$-module and hence $\mcal{T}[\infty]$ is not pure-injective as an $\Lambda$-module.

(iii): Krause shows, \cite[8.3]{KrauseGeneric}, that if $\mcal{T}$ is a generalised tube in the category of finitely presented modules over an Artin algebra then $\mcal{T}[\infty]$ is $\Sigma$-pure-injective. Since Maranda's functor commutes with direct limits and sends generalised tubes to generalised tubes, if $\mcal{T}=(M_i,f_i,g_i)_{i\in\N_0}$ is a generalised tube in $\Latt_\Lambda$ then $\mcal{T}_k[\infty]=\mcal{T}[\infty]/\mcal{T}[\infty]\pi^k$. Thus $\mcal{T}[\infty]/\mcal{T}[\infty]\pi^k$ is $\Sigma$-pure-injective.

(iv): Follows directly from (i), (iii) and \ref{pihullquotientspi}.
\end{proof}

%\begin{theorem}
%Let $\mcal{T}=(M_i,f_i,g_i)_{i\in\N_0}$ be a generalised tube in $\Latt_\Lambda$. Then the pure-injective hull of $\mcal{T}[\infty]$ is $\mcal{T}[\infty]^*$.
%\end{theorem}

When $R$ is complete and $Q\Lambda$ is a separable $Q$-algebra, the category of $\Lambda$-lattices has almost split sequences (see \cite{RogSchARS}). A stable tube is an Auslander-Reiten component of the form $\Z A_{\infty}/\tau^n$ and we call $n$ the rank of the tube. Explicitly, a stable tube of rank $n$ has points $S_i[j]$ for $1\leq i\leq n$ and $j\in\N$.  We read the index $i$ $\mod \ n$. For all $i,j\in\N$, a stable tube has a single (trivially valued) arrow $S_i[j]\rightarrow S_i[j+1]$ and a single (trivially valued) arrow $S_i[j+1]\rightarrow S_{i+1}[j]$. We will identify the points with (the isomorphism type of) the $\Lambda$-lattice they represent. As for Artin algebras, generalised tubes can be constructed from stable tubes using the following two facts.

\begin{itemize}
\item If $A,B,C\in\Latt_\Lambda$ are indecomposable and pairwise non-isomorphic and, $u:A\rightarrow B$ and $v:A\rightarrow C$ are irreducible morphisms then there is $w:A\rightarrow D$ such that $(u \ v \ w)^T:A\rightarrow B\oplus C\oplus D$ is left minimal almost split.
\item If $u:S_i[j]\rightarrow S_{i}[j+1]$ is an irreducible map, $w:S_i[j]\rightarrow W$ and $W\in\Latt_\Lambda$ is indecomposable and is not isomorphic to any of $S_i[j],S_{i+1}[j-1],\ldots,S_{i+(j-1)}[1]$ then there exists $\gamma:S_i[j+1]\rightarrow W$ such that $w=\gamma u$.
\end{itemize}

Krause \cite[9.1]{KrauseGeneric} showed that if $\mcal{T}$ is a stable tube (of rank $n$) in the module category of an Artin algebra, with the labelling of modules as above, then for each $1\leq i\leq n$, the direct limit $\varinjlim S_i[j]$ is an indecomposable pure-injective. For stable tubes in categories of lattices we know, \ref{gentubecatlatt}, that $\oplus_{i=1}^n\varinjlim S_i[j]$ has pure-injective hull $(\oplus_{i=1}^n\varinjlim S_i[j])^*$. Hence, the pure-injective hull of $\varinjlim S_i[j]$ is $(\varinjlim S_i[j])^*$.  This raises the following question.

\begin{question}
Let $R$ be a complete discrete valuation domain with field of fractions $Q$ and let $\Lambda$ be an order $R$ such that $Q\Lambda$ is a separable $Q$-algebra. If $T$ is a direct limit up a ray of irreducible monomorphisms in a stable tube in $\Latt_\Lambda$ then is $T^*$ indecomposable?
\end{question}

We are able to answer this question positively for the $\widehat{\Z_2}$-order $\Gamma:=\widehat{\Z_2}C_2\times C_2$. The torsion-free part of the Ziegler spectrum of $\Gamma$ was described in \cite{TFpartkleinfour}. However, the points were not described as modules.

We start by explaining the set up. Let $e_1,e_2,e_3,e_4$ be the primitive orthogonal idempotents as in \cite{TFpartkleinfour}. Using these idempotents, Butler \cite{ButlerKleinfour}, defined a full functor $\Delta$ from the category of $b$-reduced $\Gamma$-lattices to the category of finite-dimensional vector spaces over $\mathbb{F}_2$ with $4$ distinguished subspaces. A $\widehat{\Z_2}$-torsion-free $\Gamma$-module $M$ is \textbf{$b$-reduced} if $M\cap Me_i=2Me_i$ for all $1\leq i\leq 4$. Note that, since $e_i\notin \widehat{\Z_2}C_2\times C_2$, $Me_i$ and $M2e_i$ are calculated inside $\widehat{\Q_2}M$. Puninski and Toffalori extended this functor to the category of $b$-reduced $\widehat{\Z_2}$-torsion-free modules and showed, \cite[5.4]{TFpartkleinfour}, that it is full on $\widehat{\Z_2}$-torsion-free $b$-reduced pure-injective $\Gamma$-modules.

Let $M$ be a $b$-reduced $\widehat{\Z_2}$-torsion-free $\widehat{\Z_2}C_2\times C_2$-module. Define $M^\star:=Me_1\oplus\ldots\oplus Me_4$. Then $\Delta(M):=(V;V_1,V_2,V_3,V_4)$ where $V:=M^\star/M$ and $V_i:=Me_i+M/M\cong Me_i/M\cap Me_i=Me_i/2Me_i$.

The category of finite-dimensional vector spaces over $\mathbb{F}_2$ with $4$ distinguished subspaces may be identified with a full subcategory of modules over the path algebra $\mathbb{F}_2\widetilde{D}_4$. The only indecomposable representations which are not in this full subcategory are the simple injective $\mathbb{F}_2\widetilde{D}_4$-modules. We will make this identification and consider $\Delta$ as a functor to $\Mod\text{-}\mathbb{F}_2\widetilde{D}_4$.

As observed by Puninski and Toffalori, just from the construction, one can see that $\Delta$ is an interpretation functor. Note that if $M$ is $b$-reduced and $\widehat{\Z_2}$-torsion-free then $\Delta(M)=0$ if and only if $M$ is $\widehat{\Z_2}$-divisible.

Dieterich, in \cite{ARquiversDieterich}, showed that $\Delta$ induced an isomorphism from the Auslander-Reiten quiver of $\mathbb{F}_2\widetilde{D}_4$ with all projective points removed and all simple injective modules removed and the Auslander-Reiten quiver of $\Latt_{\Gamma}$ restricted to the b-reduced lattices. Using this he was able, \cite[3.4]{ARquiversDieterich}, to compute the full Auslander-Reiten quiver of $\Latt_{\widehat{\Z_2}C_2\times C_2}$. Moreover, see the proof of \cite[2.2]{ARquiversDieterich} and \cite[3.4]{ARquiversDieterich}, $\Delta$ induces a bimodule isomorphism between $\text{Irr}_{\Latt_{\Gamma}}(M,L)$ and $\text{Irr}_{\mathbb{F}_2\widetilde{D}_4}(\Delta(M),\Delta(L))$ for all $L,M$ indecomposable $b$-reduced $\Gamma$-lattices. In particular, $\Delta$ sends irreducible morphisms between indecomposable $b$-reduced $\Gamma$-lattices to irreducible morphisms in $\mod\text{-}\mathbb{F}_2\widetilde{D}_4$. This implies that the Auslander-Reiten quiver of $\Latt_{\widehat{\Z_2}C_2\times C_2}$ has infinitely many stable tubes of rank $1$ and $3$ stable tubes of rank $2$ and $\Delta$ sends each stable tube in $\Latt_{\widehat{\Z_2}C_2\times C_2}$ to a stable tube in $\mod\text{-}\mathbb{F}_2\widetilde{D}_4$.

Keeping our notation as above, let $S_i[j]$ be the lattices in a stable tube of rank $n=1$ or $n=2$ in $\Latt_{\widehat{\Z_2}C_2\times C_2}$. Fix $1\leq i\leq n$ and for each $j\in\N$ let $w_j:S_i[j]\rightarrow S_i[j+1]$ be an irreducible map. Let $S_i[\infty]:=\varinjlim S_i[j]$ be the direct limit along the maps $w_j$. Then $\Delta S_i[\infty] =\varinjlim \Delta S_i[j]$ is pure-injective and indecomposable by \cite[9.1]{KrauseGeneric} since $\Delta$ sends stable tubes to stable tubes. Since $\Delta$ is full on pure-injective modules, by \cite[3.15 \& 3.16]{Interpretingmodules}\footnote{The proof of the required part of \cite[3.16]{Interpretingmodules} is a little unclear. Lemma \ref{onevariable} clears this up.}, it preserves pure-injective hulls. Thus $\Delta(S_i[\infty])\cong\Delta(S_i[\infty]^*)$. Since $S_i[\infty]^*$ is reduced and $\Delta(S_i[\infty]^*)$ is indecomposable, $S_i[\infty]^*$ is indecomposable.

So finally, for each quasi-simple $S$ at the base of a tube (i.e. $S_i[1]$ for some stable tube), the $S$-pr\"{u}fer point in \cite[6.1]{TFpartkleinfour} is $S[\infty]^*$ where $S[\infty]$ is the direct limit up a ray of irreducible monomorphisms starting at $S$.

The module $T$ in question 6.2 of \cite{TFpartkleinfour} is indecomposable but not pure-injective however its pure-injective hull is indecomposable (and $\widehat{\Z_2}$-reduced).

\section{Duality}\label{S-duality}

Throughout this section, let $R$ be a Dedekind domain, $Q$ its field of fractions, $\Lambda$ an $R$-order and $Q\Lambda$ a separable $Q$-algebra. The main aim of this section is to show that the lattice of open sets of $\Zg_\Lambda^{tf}$ is isomorphic to the lattice of open sets of ${_\Lambda}\Zg^{tf}$. We will also show, by other methods, that the m-dimension of $\pp^1_\Lambda(\Tf_{\Lambda})$ is equal to the m-dimension of ${_\Lambda}\pp^1({_\Lambda}\Tf)$ and that the Krull-Gabriel dimension of $(\Latt_\Lambda,\Ab)^{fp}$ is equal to the Krull-Gabriel dimension of $({_\Lambda}\Latt,\Ab)^{fp}$.

\subsection{Duality for the $R$-reduced part of $\Zg_{\Lambda}^{tf}$ when $R$ is a discrete valuation domain}\label{dualityordersoverdvr}
Throughout this subsection $R$ will be a discrete valuation domain, $k$ will be a natural number strictly greater than Maranda's constant for $\Lambda$ as an $R$-order and $I:\Tf_\Lambda\rightarrow \Mod\text{-}\Lambda_k$ (respectively $I:{_\Lambda}\Tf\rightarrow \Lambda_k\text{-}\Mod$) will be Maranda's functor.

Maranda's functor $I:\Tf_\Lambda\rightarrow \Mod\text{-}\Lambda_k$ is an interpretation functor. The kernel of $I$ is the definable subcategory of $R$-divisible modules. Since $Q\Lambda$ is separable, by \ref{divimpliesinj} and the discussion just below it, all indecomposable pure-injective modules in $\Tf_\Lambda$ are either $R$-reduced or $R$-divisible modules. When $\Lambda$ is an order over a discrete valuation domain $R$, we will write $\Zg^{rtf}_\Lambda$ for the subset of $R$-reduced modules in $\Zg_\Lambda^{tf}$.  We have shown in section \ref{SMarfun} that if $N,M\in\Tf_\Lambda$ are $R$-reduced and pure-injective then $IN\cong IM$ implies $N\cong M$ and that if $N$ is also indecomposable then so is $IN$. Thus \ref{superintZg} gives us the following theorem.

\begin{theorem}\label{Marhomeo}
The map which sends $N\in \Zg_\Lambda^{rtf}$ to $N/N\pi^k\in\Zg_{\Lambda_k}$ induces a homeomorphism onto its image which is closed.
\end{theorem}

In theory, the above theorem could be used to give a description of $\Zg_\Lambda^{rtf}$ and hence $\Zg_\Lambda^{tf}$ based on a description of $\Zg_{\Lambda_k}$. But, as explained in the introduction to this paper, $\Zg_{\Lambda_k}$ is generally much more complicated than $\Zg_\Lambda^{tf}$.

Based on Prest's duality for pp formulas, Ivo Herzog defined a lattice isomorphism between the lattice of open subsets of $\Zg_S$ and the lattice of open subsets of ${_S}\Zg$.

\begin{theorem}\cite{Herzogduality}
There is a lattice isomorphism $D$ between that lattice of open subsets of $\Zg_S$ (respectively ${_S}\Zg$) and the lattice of open subsets of ${_S}\Zg$ (respectively ${_S}\Zg$) which is given on basic open sets by  \[\left(\phi/\psi\right)\mapsto \left(D\psi/D\phi\right)\] for $\phi,\psi$ pp-$1$-formulas. Moreover $D^2$ is the identity map.
\end{theorem}

It is unknown if this lattice isomorphism is always induced by a homeomorphism.

If $X$ is a closed subset of $\Zg_S$ then we will write $DX$ for ${_S}\Zg\backslash D(\Zg_S\backslash X)$. Since closed subsets of $\Zg_S$ are in correspondence with the definable subcategories of $\Mod\text{-}S$, this isomorphism also defines an inclusion preserving bijection between the definable subcategories of $\Mod\text{-}S$ and $S\text{-}\Mod$. If $\mcal{X}\subseteq \Mod\text{-}S$ is a definable subcategory then we will write $D\mcal{X}$ for the corresponding definable subcategory of $S\text{-}\Mod$.

Herzog's duality can be applied to closed subspaces of $\Zg_S$ as follows. Let $X$ be a closed subset of $\Zg_S$. Open subsets of  $\Zg_S$ containing $\Zg_S\backslash X$ are in bijective correspondence with open subsets of $X$ equipped with the subspace topology via the map $U\mapsto U\cap X$. If $U$ is an open subset of $\Zg_S$ containing $\Zg_S\backslash X$ then $DU$ is an open subset of ${_S}\Zg$ containing ${_S}\Zg\backslash DX$. Thus $D$ induces a lattice isomorphism between the lattice of open sets of $X$ and the lattice of open sets of $DX$ both equipped with the appropriate subspace topology.

Herzog's isomorphism $D$ sends the definable subcategory $\Tf_\Lambda$ to the definable subcategory of $R$-divisible $\Lambda$-modules. Thus, directly applying Herzog's duality does not give an isomorphism between the open subsets of $\Zg_\Lambda^{tf}$ and ${_\Lambda}\Zg^{tf}$. With this in mind, we instead use the right module version of Maranda's functor $I$ to move to $\Mod\text{-}\Lambda_k$, we then apply $D$ there and then use the left module version of Maranda's functor to move back to ${_\Lambda}\Tf$. This will give us an isomorphism between the lattice of open subsets of $\Zg_\Lambda^{rtf}$ and ${_\Lambda}\Zg^{rtf}$.

%For any ring $S$, Herzog defined an isomorphism $D$ between the lattice of open subsets of $\Zg_S$ and ${_S}\Zg$. Since closed subsets of $\Zg_S$ are in correspondence with the definable subcategories of $\Mod\text{-}S$, this isomorphism also defines an inclusion preserving bijection between the definable subcategories of $\Mod\text{-}S$ and $S\text{-}\Mod$.
%
%Herzog's isomorphism $D$ sends the definable subcategory $\Tf_\Lambda$ to the definable subcategory of $R$-divisible $\Lambda$-modules. Thus, directly applying Herzog's duality does not give an isomorphism between the open subsets of $\Zg_\Lambda^{tf}$ and ${_\Lambda}\Zg^{tf}$. With this in mind and supposing that $R$ is a discrete valuation domain, we instead use the right module version of Maranda's functor $I$ to move to $\Mod\text{-}\Lambda_k$, we then apply $D$ there and then use the left module version of Maranda's functor to move back to $\Tf_\Lambda$. This will give us an isomorphism between the lattice of open subsets of $\Zg_\Lambda^{rtf}$ and ${_\Lambda}\Zg^{rtf}$.

Our first step is to show that $\langle I\Tf_\Lambda\rangle=D\langle I{_\Lambda}\Tf\rangle$.

The contravariant functor
\[\Hom_R(-,R):\Mod\text{-}\Lambda\rightarrow \Lambda\text{-}\Mod\] induces an equivalence between the category of right $\Lambda$-lattices and the opposite of the category of left $\Lambda$-lattices, see \cite[sect. IX 2.2]{rogordersII}. If $M$ is right $\Lambda$-lattice, denote the left $\Lambda$-lattice,  $\Hom_R(M,R)$ by $M^{\dagger}$.

The ring $\Lambda/\pi^n\Lambda$ is a $R/\pi^nR$-Artin algebra. For all $S$-Artin algebras $\mcal{A}$, there is a duality between $\mod\text{-}\mcal{A}$ and $\mcal{A}\text{-}\mod$ given by $\Hom_S(-,E)$ where $E$ is the injective hull of $S/\rad(S)$. We will write $M^*$ for $\Hom(M,E)$. If $S=R/\pi^nR$ then $S/\rad(S)=R/\pi R$ and $E=R/\pi^nR$.

We will now show that if $L$ is a right $\Lambda$-lattice then $(IL)^*=IL^\dagger$.

\begin{lemma}\label{connectstardagger}
If $M$ is a right $\Lambda$-lattice and $n\in\N$ then
\[\Hom_R(M,R)/\pi^n\Hom_R(M,R)\cong \Hom_{R/\pi^n}(M/M\pi^n,R/\pi^nR).\]
\end{lemma}
\begin{proof}
For $f\in \Hom_R(M,R)$, let $\overline{f}:M/M\pi^n\rightarrow R/\pi^nR \in \Hom_{R/\pi^nR}(M/M\pi^n,R/\pi^nR)$ be the homomorphism which sends $m+M\pi^n$ to $f(m)+\pi^nR$.

Let $\Phi:\Hom_R(M,R)\rightarrow \Hom_{R/\pi^n}(M/M\pi^n,R/\pi^nR)$ be defined by $\Phi(f)=\overline{f}$. It is clear that $\Phi$ is a homomorphism of left $\Lambda$-modules. Since $M$ is projective as an $R$-module, $\Phi$ is surjective.

If $\Phi(f)=0$ then for all $m\in M$, $\pi^n|f(m)$. For all $m\in M$, let $g(m)\in M$ be such that $g(m)\pi^n=f(m)$. Since $M$ is $R$-torsion-free, the choice of $g(m)$ is unique. From this is follows easily that $g$ is a homomorphism of $R$-modules. Thus, if $\Phi(f)=0$ then $f\in\pi^n\Hom_R(M,R)$.
\end{proof}

The next remark follows from the fact, see \cite[1.3.13]{PSL} for instance, that if $\mcal{A}$ is an Artin algebra, $\phi/\psi$ is a pp-pair and $M$ is a finite length $\mcal{A}$-module then $\phi(M)=\psi(M)$ if and only if $D\phi(M^*)=D\psi(M^*)$.

\begin{remark}\label{dualityartindefsubcats}
Suppose that $\mcal{A}$ is an Artin algebra and $\{M_i \st i\in I\}$ is a set of finite length right $\mcal{A}$-modules. Then
\[D\langle M_i\st i\in I\rangle = \langle M_i^*\st i\in I\rangle.\]
\end{remark}

\begin{lemma}\label{leftimrightimmaranda}
The following equalities hold.
\begin{eqnarray}
  \langle I\Tf_{\Lambda}\rangle &=& \langle IL\st L \text{ is an indecomposable right } \Lambda \text{-lattice}\rangle \\
   &=& \langle IM^{\dagger}\st M \text{ is an indecomposable left } \Lambda \text{-lattice}\rangle \\
   &=& \langle (IM)^*\st M \text{ is an indecomposable left } \Lambda \text{-lattice}\rangle \\
   &=& D\langle IM\st M \text{ is an indecomposable left } \Lambda \text{-lattice}\rangle \\
   &=& D\langle I{_\Lambda}\Tf\rangle
\end{eqnarray}
\end{lemma}
\begin{proof}

\noindent
\begin{enumerate}
\item This holds because all $N\in \Tf_{\Lambda}$ are direct limits of $\Lambda$-lattices, all $\Lambda$-lattices are direct sums of indecomposable $\Lambda$-lattices and $I$ commutes with direct limits.
\item For all (right) $\Lambda$-lattices $L^{\dagger\dagger}\cong L$ and $L^{\dagger}$ is a (left) $\Lambda$-lattice.
\item \ref{connectstardagger}
\item \ref{dualityartindefsubcats}
\item Same as (1).
\end{enumerate}
\end{proof}

Herzog's duality $D$ gives an isomorphism from the lattice of open sets of $\Zg(\langle I\Tf_{\Lambda}\rangle)$ to the lattice of open sets of $\Zg(D\langle I\Tf_{\Lambda}\rangle)$. By \ref{leftimrightimmaranda}, $D\langle I\Tf_{\Lambda}\rangle=\langle I{_\Lambda}\Tf\rangle$.

If $U$ is an open subset of $\Zg_\Lambda^{rtf}$ (resp. ${_\Lambda}\Zg^{rtf}$) then write $IU$ for the set of all $IN$ where $N\in U$.

\begin{definition}
Let $U$ be an open subset of $\Zg_\Lambda^{rtf}$. Define
\[dU:=\{N\in {_\Lambda}\Zg^{rtf} \st IN\in DIU \}. \]
\end{definition}

By \ref{Marhomeo}, $IU$ is an open subset of $\Zg(\langle I\Tf_{\Lambda}\rangle)$. So $DIU$ is an open subset of $\Zg(\langle I{_\Lambda}\Tf\rangle)$. Again by \ref{Marhomeo}, the set of $N\in {_\Lambda}\Zg^{rtf}$ such that $IN\in DIU$ is an open subset of ${_\Lambda}\Zg^{rtf}$.

\begin{proposition}\label{lattopenisoonreducedpoints}
The map $d$ between the lattice of open sets of $\Zg_\Lambda^{rtf}$ and ${_\Lambda}\Zg^{rtf}$ is a lattice isomorphism.
\end{proposition}
\begin{proof}
The homeomorphism from \ref{Marhomeo} sends an open subset $U$ of $\Zg_\Lambda^{rtf}$ to $IU\subseteq \Zg(\langle I\Tf_\Lambda\rangle)$. So the map sending $U$ to $IU$ is a lattice isomorphism. By \ref{leftimrightimmaranda}, Herzog's duality gives a lattice isomorphism between the open subsets of $\Zg(\langle I\Tf_\Lambda\rangle)$ and the lattice of open subset of  $\Zg(\langle I{_\Lambda}\Tf\rangle)$. Thus the map which sends an open subset $U$ of $\Zg_\Lambda^{rtf}$ to $DIU\subseteq \Zg(\langle I{_\Lambda}\Tf\rangle)$ is a lattice isomorphism. Finally the inverse of the of the homeomorphism from \ref{Marhomeo} sends an open subset of $W\subseteq \Zg(\langle I{_\Lambda}\Tf\rangle)$ to the set of all $N\in {_\Lambda}\Zg^{rtf}$ such that $IN\in W$. So this map is also a lattice isomorphism. Since $d$ is the composition of these three lattice isomorphisms, $d$ is also a lattice isomorphism.
%
%The Maranda functor $I$ gives a homeomorphism between $\Zg_\Lambda^{rtf}$ (resp. ${_\Lambda}\Zg^{rtf}$) and  $\Zg(\langle I\Tf_\Lambda\rangle)$ (resp. $\Zg(\langle I{_\Lambda}\Tf\rangle)$) by sending $N\in\Zg_\Lambda^{rtf}$ (resp. $N\in {_\Lambda}\Zg^{rtf}$) to $IN\in\Zg(\langle I\Tf_\Lambda\rangle)$ (resp. $IN\in\Zg(\langle I{_\Lambda}\Tf\rangle)$) by \ref{superintZg}, \ref{marandapresind} and \ref{marandapresisotype}.
%
%If $U$ is an open subset of $\Zg_\Lambda^{rtf}$ then $IU$ is an open subset of $\Zg(\langle I\Tf_\Lambda\rangle)$. So $DIU$ is an open subset of $\Zg(\langle I{_\Lambda}\Tf\rangle)$. Define
%\[dU:=\{N\in {_\Lambda}\Zg^{rtf} \st IN\in DIU \} \}.\]
\end{proof}

If $\Lambda$ is an order over a complete discrete valuation domain then the $\Lambda$-lattices are pure-injective (see \cite[2.2]{TfpartZgorders} for instance). When $R$ is not complete, we can instead consider the lattices over the $\widehat{R}$-order $\widehat{\Lambda}$. Then the  $\widehat{\Lambda}$-lattices are pure-injective as $\widehat{\Lambda}$-modules and hence also as $\Lambda$-modules. Moreover, if $L$ is an indecomposable $\widehat{\Lambda}$-lattice then, since $L$ is $R$-reduced, $L$ is also indecomposable as a $\Lambda$-module (see \cite[Remark 1 p 1130]{TFpartZgRG} for a proof over group rings that also works in our context).

\begin{proposition}
Let $R$ be a discrete valuation domain and $\Lambda$ an $R$-order. If $L$ is an indecomposable right $\widehat{\Lambda}$-lattice then for all open sets $U\subseteq \Zg_\Lambda^{rtf}$, $L\in U$ if and only if $L^{\dagger}\in dU$ where $L^\dagger:=\Hom_{\widehat{R}}(L,\widehat{R})$.
\end{proposition}
\begin{proof}
First note that $IL$ is finite-length as a $\Lambda_k$-module. Since $\Lambda_k$ is an Artin algebra, if $M\in \Zg(\langle I\Tf_{\Lambda}\rangle)$ is finite-length then for all open subsets $U$ of $\Zg(\langle I\Tf_{\Lambda}\rangle)$, $M\in U$ if and only if $M^*\in DU$, see \cite[1.3.13]{PSL}. So, if $L$ is an indecomposable right $\widehat{\Lambda}$-lattice then $L\in U$ if and only if $IL\in IU$ and $IL\in IU$ if and only if $(IL)^*\in DIU$. By \ref{connectstardagger}, $(IL)^*=IL^\dagger$, so $(IL)^*\in DIU$ if and only if $L^\dagger \in dU$. So $L\in U$ if and only if $L^\dagger\in dU$.
\end{proof}

\subsection{Duality for $\Zg_\Lambda^{tf}$}

We now work to extend \ref{lattopenisoonreducedpoints} in two ways concurrently. We extend the isomorphism to an isomorphism between the lattices of open subsets of $\Zg_\Lambda^{tf}$ and ${_\Lambda}\Zg^{tf}$ and we extend the statement to the case where $R$ is a Dedekind domain.

In order to do this we need to recall some key features of $\Zg_\Lambda^{tf}$ from \cite{TfpartZgorders}.

As explained in \cite[Section 3]{TfpartZgorders}, for each $P\in\text{Max}R$, the canonical homomorphism $\Lambda\rightarrow \Lambda_P$ induces, via restriction of scalars, an embedding of $\Zg^{tf}_{\Lambda_P}$ into $\Zg^{tf}_{\Lambda}$ and the image of this embedding is closed. We identify $\Zg^{tf}_{\Lambda_P}$ with its image. Moreover, for all $N\in\Zg_\Lambda^{tf}$, there exists a $P\in\text{Max}R$ such that $N\in \Zg^{tf}_{\Lambda_P}$. So
\[\Zg_{\Lambda}^{tf}=\bigcup_{P\in\text{Max}R}\Zg_{\Lambda_P}^{tf}.\] Finally, if $N\in\Zg_{\Lambda_P}$ for all $P\in\text{Max}R$ then $N$ is $R$-divisible and hence may be viewed as a module over $Q\Lambda$. Since $Q\Lambda$ is separable, hence semi-simple, all indecomposable $R$-divisible modules, when viewed as $Q\Lambda$-modules, are simple.

For each $P\in\text{Max}R$, let $P|x$ denote the pp formula $\exists y_1,\ldots,y_n \ x=\sum_{i=1}^ny_ir_i$ where $r_1,\ldots r_n$ generate $P$. In all $\Lambda$-modules $M$, $P|x$ defines the subset $MP$. If $P,P'\in\text{Max}R$ are not equal then $\left(x=x/P|x\right)\cap\left(x=x/P'|x\right)$ is empty. For all $N\in\Zg_\Lambda^{tf}$, either $N$ is $R$-divisible or $N\in\left(x=x/P|x\right)$ for some $P\in\text{Max}R$. So

\[\Zg_\Lambda^{tf}=\Zg_{Q\Lambda}\cup\bigcup_{P\in\text{Max}R}\left(x=x/P|x\right).\]

Note that $\left(x=x/P|x\right)=\Zg^{tf}_{\Lambda_P}\backslash\Zg_{Q\Lambda}$. Under the assumption that $Q\Lambda$ is a semi-simple $Q$-algebra, this means that $\left(x=x/P|x\right)$ is the set of $R_P$-reduced indecomposable pure-injective $\Lambda_P$-modules. For this reason, we will write $\Zg^{rtf}_{\Lambda_P}$ for this set. Note that this notation matches that of the previous section when $\Lambda$ is an order over a discrete valuation domain.

\begin{theorem}\label{ZgTF}
Let $R$ be a Dedekind domain with field of fractions $Q$, and $\Lambda$
an $R$-order such that $Q\Lambda$ is semisimple. If $N \in \Zg_\Lambda^{tf}$, then either
\begin{itemize}
\item $N$ is a simple $A$-module, or
\item there is some maximal ideal $P$ of $R$ such that $N \in \Zg^{tf}_{\widehat{\Lambda_P}}$ and $N$ is $R_P$-reduced.
\end{itemize}
Moreover, if $N\in\Zg^{tf}_{\widehat{\Lambda_P}}$ is $R_P$-reduced then $N\in \Zg_\Lambda^{tf}$.
\end{theorem}

This theorem means that if $Q\Lambda$ is separable then the $R_P$-reduced points of $\Zg_\Lambda^{tf}$ can be identified with the $\widehat{R_P}$-reduced (equivalently $R_P$-reduced) points of $\Zg^{tf}_{\widehat{\Lambda_P}}$. Following \cite{TFpartZgRG}, it is shown in \cite[3.2]{TfpartZgorders} that the topology on the set of $R_P$-reduced points of $\Zg_\Lambda^{tf}$ is the same whether it is viewed as a subspace of $\Zg_{\Lambda_P}^{tf}$ or $\Zg^{tf}_{\widehat{\Lambda_P}}$. Thus we may identify $\Zg^{rtf}_{\Lambda_P}$ and $\Zg^{rtf}_{\widehat{\Lambda_P}}$.

We have already mentioned in Section \ref{dualityordersoverdvr} that a $\widehat{\Lambda_P}$-lattice is pure-injective. Therefore the restrictions of indecomposable $\widehat{\Lambda_P}$-lattices to $\Lambda$ are points in $\Zg_\Lambda^{tf}$.

From now on, if $P\in\text{Max}R$ then let $d_P$ denote the isomorphism between the lattice of open subsets of $ \Zg_{\Lambda_P}^{rtf}$ and of ${_{\Lambda_P}}\Zg^{rtf}$ induced by $d$ for $\Lambda_P$. Patching the $d_P$ together as $P\in\text{Max}R$ varies will give us an isomorphism between the open subset of $\bigcup_{P\in\text{Max}R}\Zg_{\Lambda_P}^{rtf}\subseteq \Zg_\Lambda^{tf}$ and the open subsets of $\bigcup_{P\in\text{Max}R}{_{\Lambda_P}}\Zg^{rtf}\subseteq {_\Lambda}\Zg^{tf}$. Thus, we just need to know what to do with open subsets which contain $R$-divisible points.

Let $e_1,\ldots,e_n$ be a complete set of centrally primitive orthogonal idempotents for $Q\Lambda$. For each $1\leq i\leq n$, $e_iQ\Lambda$ is isomorphic as a right $Q\Lambda$-module to $S_i^{(\alpha_i)}$ for some simple right $Q\Lambda$-module $S_i$ and if $S_i\cong S_j$ then $i=j$.

\begin{lemma}\cite[2.7]{TfpartZgorders}\label{closedpoints}
Let $N\in\Zg_\Lambda^{tf}$ and $S\in\Zg_{Q\Lambda}$. If $S$ is a direct summand of $QN$ then $S$ is in the closure of $N$. In particular, if $N$ is a closed point in $\Zg_\Lambda^{tf}$ then $N\in\Zg_{Q\Lambda}$.
\end{lemma}

\begin{lemma}
Let $D$ be a Dedekind domain with field of fractions $Q$ and let $\Lambda$ be an order over $D$ such that $Q\Lambda$ is semisimple. Let $e\in Q\Lambda$ be a centrally primitive idempotent, $S$ the simple right $Q\Lambda$-module corresponding to $e$ and suppose that $d\in D$ is such that $ed\in\Lambda$. The following are equivalent for all $N\in\Zg_\Lambda^{tf}$.
\begin{enumerate}
\item $N\in\left(xd(1-e)=0/x=0\right)$
\item $S$ is a direct summand of $QN$
\item $S$ is in the closure of $N$
\end{enumerate}
\end{lemma}
\begin{proof}
$(1)\Rightarrow (2)$ Suppose $md(1-e)=0$ and $m\neq 0$. Then, as an element of $QN$ viewed as a $Q\Lambda$-module, $m(1-e)=0$. Thus $m=me$. The kernel of the  homomorphism from $Q\Lambda$ to $QN$ sending $1$ to $m$ contains $(1-e)Q\Lambda$ and thus induces a non-zero homomorphism from $eQ\Lambda$ to $QN$. Thus $S$ is a submodule and hence a direct summand of $QN$.

$(2)\Rightarrow (3)$ This is \ref{closedpoints}.

$(3)\Rightarrow (1)$ Suppose $S$ is in the closure of $N$. Since $eQ\Lambda(1-e)d=0$, $S \in \left(xd(1-e)=0/x=0\right)$. Thus $N\in\left(xd(1-e)=0/x=0\right)$.
%
%$(1)\Rightarrow (2)$ Suppose $md(1-e)=0$ and $m\neq 0$. Then, as an element of $QN$ viewed as a $Q\Lambda$-module, $m(1-e)=0$. Thus $m=me$. The kernel of the  homomorphism from $Q\Lambda$ to $QN$ sending $1$ to $m$ contains $(1-e)Q\Lambda$ and thus induces a non-zero homomorphism from $eQ\Lambda$ to $QN$. Thus $S$ is a submodule and hence a direct summand of $QN$. So by \cite[2.7]{TfpartZgorders}, $S$ is in the closure of $N$.
%
%$(2)\Rightarrow (1)$ Suppose $S$ is in the closure of $N$. Since $eQ\Lambda(1-e)d=0$, $S \in \left(xd(1-e)=0/x=0\right)$. Thus $N\in\left(xd(1-e)=0/x=0\right)$.
\end{proof}

Note that the above shows that the set of points specialising to a closed point in $\Zg_\Lambda^{tf}$ is an open set. For $S\in \Zg_{Q\Lambda}$, we will write $\mcal{V}(S)$ for the open set of points whose closure contains $S$.

\goodbreak
\begin{cor}\label{canformopenset}
Let $U$ be an open subset of $\Zg_{\Lambda}^{tf}$. Then

\[U=\bigcup_{P}(U\cap \Zg^{rtf}_{\Lambda_P})\cup\bigcup_{S\in \lambda(U)}\mcal{V}(S)\] where $\lambda(U):=\{S\in\Zg_{Q\Lambda} \st S \in U\}$.

\end{cor}
\begin{proof}
If $N\in \Zg_{\Lambda}^{tf}$ then either $N\in\Zg^{rtf}_{\Lambda_P}$ for some $P\in \text{Max}R$ or $N\in\Zg_{Q\Lambda}$. So, since for all $S\in \Zg_{Q\Lambda}$, $S\in\mcal{V}(S)$, $U\subseteq \bigcup_{P}(U\cap \Zg^{rtf}_{\Lambda_P})\cup\bigcup_{S\in \lambda(U)}\mcal{V}(S)$.

Suppose $S\in\lambda(U)$ and $N\in\mcal{V}(S)$. Then $S$ is in the closure of $N$. Hence $N\in U$. Thus $\mcal{V}(S)\subseteq U$. So $U\supseteq \bigcup_{P}(U\cap \Zg^{rtf}_{\Lambda_P})\cup\bigcup_{S\in \lambda(U)}\mcal{V}(S)$.
\end{proof}

For each simple $Q\Lambda$-module $S$, we now consider where to send the open set $\mcal{V}(S)$. In particular, we need to calculate the image of $\mcal{V}(S)\cap\Zg^{rtf}_{\Lambda_P}$ under $d_P$ for each $P\in\text{Max}R$.

\begin{lemma}\label{daggerdivpoints}
Let $R$ be a discrete valuation domain and $\Lambda$ an $R$-order. For all $M\in\Latt_\Lambda$, $Q\Hom_R(M,R)$ and $\Hom_Q(MQ,Q)$ are isomorphic as $Q\Lambda$-modules.
\end{lemma}
\begin{proof}
Let $\Delta:\Hom_R(M,R)\rightarrow \Hom_Q(MQ,Q)$ be defined by setting $\Delta(f)(m\cdot q)=f(m)\cdot q$ for all $m\in M$ and $q\in Q$. A quick computation shows that for all $f\in \Hom_R(M,R)$, $\Delta(f)$ is a well-defined element of $\Hom_Q(MQ,Q)$ and $\Delta$ is an injective homomorphism of left $\Lambda$-modules. Since $\Hom_Q(MQ,Q)$ is $Q$-divisible, $\Delta$ extends to an injective homomorphism $\Delta'$ from $Q\Hom_R(M,R)$ to $\Hom_Q(MQ,Q)$.

Suppose that $M$ is rank $n$. Then $\dim_Q MQ=\dim_Q\Hom_Q(MQ,Q)=\dim_Q Q\Hom_R(M,R)=n$. Thus $\Delta'$ is an injective homomorphism between two $n$-dimensional $Q$-vector spaces and hence is surjective.
\end{proof}

\begin{lemma}\label{opensubsetsforsimples}
Let $R$ be a discrete valuation domain. Let $L\in \text{Latt}_\Lambda$, $e$ a central idempotent of $Q\Lambda$ and $d\in R$ be such that $ed\in \Lambda$. Then $L\in\left(x(e-1)d=0/x=0\right)$ if and only if $L^\dagger\in\left((e-1)dx=0/x=0\right)$.
\end{lemma}
\begin{proof}
Suppose $L\in\left(x(e-1)d=0/x=0\right)$. Then there exists $a\in QL\backslash\{0\}$ such that $a(e-1)=0$. By \ref{daggerdivpoints}, $Q\Hom_R(L,R)\cong\Hom_Q(QL,Q)$. Thus we need to show that there exists $0\neq f\in \Hom_Q(QL,Q)$ such that $(e-1)\cdot f=0$. Since $e$ is central, $QL=QLe\oplus QL(e-1)$ and $QLe\neq 0$. Take $f\in \Hom_Q(QL,Q)$ such that $f$ is zero on $QL(e-1)$ and non-zero on $QLe$. Then for all $m\in QL$, $(e-1)\cdot f(m)=f(m(e-1))=0$ but $f\neq 0$. Thus there exists $b\in QL^\dagger\backslash\{0\}$ such that $(e-1)\cdot b=0$. There exists $r\in R\backslash\{0\}$ such that $rb\in L^\dagger$ and $(e-1)d\cdot rb=0$. So $L^\dagger\in\left((e-1)dx=0/x=0\right)$.
\end{proof}

\begin{lemma}\label{denseforduality}
Let $a\in \Lambda$. The set of indecomposable $\widehat{\Lambda_P}$-lattices, as $P\in\text{Max}R$ varies, is dense in $\Zg_\Lambda^{tf}\backslash\left(xa=0/x=0\right)$.
\end{lemma}
\begin{proof}
Suppose that $\left(\phi/\psi\right)\cap(\Zg_\Lambda^{tf}\backslash\left(xa=0/x=0\right))\neq\emptyset$. Pick $N\in \left(\phi/\psi\right)\cap(\Zg_\Lambda^{tf}\backslash\left(xa=0/x=0\right))$. Since $N$ is a direct union of its finitely generated submodules, there exists a finitely generated submodule $L$ of $N$ such that $\phi(L)\supsetneq \psi(L)$. Since $L$ is a submodule of $N$, $L$ is $R$-torsionfree and $\ann_La=0$. Thus $\phi(H(L))\supsetneq\psi(H(L))$ and $\ann_{H(L)}a=0$. Since $H(L)$ is isomorphic to $\prod_{P\in\text{Max}R}\widehat{L_P}$ by \ref{pihulllattices}, for all $P\in\text{Max}R$, $\ann_{\widehat{L_P}}a=0$ and there exists $P\in\text{Max}R$ such that $\phi(\widehat{L_P})\supsetneq\psi(\widehat{L_P})$. Thus there exists a $P\in\text{Max}R$ and a $\widehat{\Lambda_P}$-lattice $M$ such that $\phi(M)\supsetneq\psi(M)$ and $\ann_Ma=0$. Since the category of $\widehat{\Lambda_P}$-lattices is Krull-Schmidt, it follows that there exists an indecomposable $\widehat{\Lambda_P}$-lattice with the required properties.
\end{proof}

The following is proved in the case that $R$ is a discrete valuation domain in \cite{TFpartZgRG}.

\begin{cor}
The set of indecomposable $\widehat{\Lambda_P}$-lattices, as $P\in\text{Max}R$ varies, is a dense subset of $\Zg^{tf}_{\Lambda}$. Therefore, all isolated points in $\Zg_\Lambda^{tf}$ are $\widehat{\Lambda_P}$-lattices for some $P\in\text{Max}R$.
\end{cor}
\begin{proof}
Density is a special case of \ref{denseforduality}. It is shown in \cite[2.4]{TfpartZgorders} that the indecomposable $\widehat{\Lambda_P}$-lattices are isolated in $\Zg^{tf}_{\widehat{\Lambda_P}}$. As explained just after \ref{ZgTF}, we may identify $\Zg^{rtf}_{\widehat{\Lambda_P}}$ with $\Zg^{rtf}_{\Lambda_P}$. Thus the $\widehat{\Lambda_P}$-lattices are isolated in $\Zg^{rtf}_{\Lambda_P}$. Finally, viewed as a subspace of $\Zg^{tf}_\Lambda$, $\Zg^{rtf}_{\Lambda_P}$ is equal to the open set $\left(x=x/P|x\right)$. Thus the indecomposable $\widehat{\Lambda_P}$-lattices are isolated in $\Zg^{tf}_{\Lambda}$.
\end{proof}
%\begin{proof}
%Suppose that $\phi/\psi$ is a pp-pair and $N\in\Zg_\Lambda^{tf}$ is in $\left(\phi/\psi\right)$. By standard arguments there exists a finitely generated submodule $L$ of $N$ such that $\phi(L)\supsetneq \psi(L)$. Since $L$ is $R$-torsion-free, $L$ is a $\Lambda$-lattice. The pure-injective hull $H(L)$ of $L$ is $\prod_{P\in\text{Max}R}\widehat{L_P}$. So there is some $P\in\text{Max}R$ such that $\phi(\widehat{L_P})\supsetneq \psi(\widehat{L_P})$. Since the category of $\widehat{\Lambda_P}$-lattices is Krull-Schmidt, there is an indecomposable direct summand $M$ of $\widehat{L_P}$ as a $\widehat{\Lambda_P}$-module such that $\phi(M)\supsetneq \psi(M)$.
%\end{proof}

%Suppose that $S$ is a simple right $Q\Lambda$-module and $S$ is a direct summand of $MQ$. Then the simple left $Q\Lambda$-module $\Hom_Q(S,Q)$ is a direct summand of $Q\Hom_R(M,R)$. Moreover, if $S$ is the simple right $Q\Lambda$-module corresponding to the centrally primitive idempotent $e$ then $\Hom_Q(S,Q)$ is the simple left $Q\Lambda$-module corresponding to $e$.

Recall that for each $P\in\text{Max} R$, $d_P$ is the isomorphism between the lattice of open subsets of $ \Zg_{\Lambda_P}^{rtf}$ and of ${_{\Lambda_P}}\Zg^{rtf}$ defined in section \ref{dualityordersoverdvr}.

\begin{lemma}\label{dPonsimples}
For all simple $Q\Lambda$-modules $S$ and all $P\in\text{Max}R$, \[d_P(\mcal{V}(S)\cap\Zg^{rtf}_{\Lambda_P})=\mcal{V}(S^*)\cap {_{\Lambda_P}}\Zg^{rtf}.\]
\end{lemma}
\begin{proof}
We first show that if $L$ is an indecomposable right $\widehat{\Lambda_P}$-lattice and $S$ is a simple right $Q\Lambda$-module then $L\in\mcal{V}(S)$ if and only if $L^\dagger\in\mcal{V}(S^*)$. Let $e$ be a centrally primitive idempotent of $Q\Lambda$ corresponding to $S$. Note that $e$ central and idempotent as an element of $\widehat{Q_P}\widehat{\Lambda}$. We have shown in \ref{opensubsetsforsimples} that $L\in\left(x(e-1)d=0/x=0\right)$ if and only if $L^\dagger\in\left((e-1)dx=0/x=0\right)$. So it is enough to show that $\left((e-1)dx=0/x=0\right)=\mcal{V}(S^*)$. But this is clear because certainly $(e-1)S^*=0$ and thus $e$ is a centrally primitive idempotent corresponding to $S$.

Since, by \ref{denseforduality}, the indecomposable right $\widehat{\Lambda_P}$-lattices are dense in the closed subset $\Zg_{\Lambda_P}^{rtf}\backslash\left(x(e-1)d=0/x=0\right)$ of $\Zg_{\Lambda_P}^{rtf}$,
\[\Zg_{\Lambda_P}^{rtf}\backslash\left(x(e-1)d=0/x=0\right)\subseteq \Zg_{\Lambda_P}^{rtf}\backslash d_P(\left((e-1)dx=0/x=0\right)\cap {_{\Lambda_P}}\Zg^{rtf}).\]
So $d_P(\left((e-1)dx=0/x=0\right)\cap {_{\Lambda_P}}\Zg^{rtf})\subseteq \left(x(e-1)d=0/x=0\right)$. The same argument using left $\widehat{\Lambda_P}$-lattices shows that
\[d_P(\left(x(e-1)d=0/x=0\right)\cap \Zg_{\Lambda_P}^{rtf})\subseteq \left((e-1)dx=0/x=0\right).\]

So, since $d_P^2$ is the identity, \[d_P(\left(x(e-1)d=0/x=0\right)\cap\Zg_{\Lambda_P}^{rtf})=\left((e-1)dx=0/x=0\right)\cap {_{\Lambda_P}}\Zg^{rtf}.\]
%
%
%Since $L$ is a $\widehat{\Lambda_P}$-lattice, $QL$ and $\widehat{Q_P}L$ are isomorphic as $Q\Lambda$-module by REF. By REF,
%
%
%Note that $L\in\mcal{V}(S)$ if and only if $S$ is a direct summand of $QS$. Suppose that $S$ is a direct summand of $QL$.
%
%
%If $L$ is an indecomposable right $\Lambda$-lattice then $L\in \left(x\pi^k(1-e_i)=0/x=0\right)$ if and only if $L^{\dagger}\in \left(\pi^k(1-e_i)x=0/x=0\right)$ by \ref{specclosedpointopen} and \ref{daggerdivpoints}. By \ref{lattopenisoonreducedpoints}, $L\in \left(x\pi^k(1-e_i)=0/x=0\right)$ if and only if $L^\dagger\in d(\left(x\pi^k(1-e_i)=0/x=0\right)\cap \Zg^{tf,red}_{\Lambda})$. Thus,
%$L^\dagger\in d(\left(x\pi^k(1-e_i)=0/x=0\right)\cap \Zg^{tf,red}_{\Lambda})$ if and only if $L^{\dagger}\in \left(\pi^k(1-e_i)x=0/x=0\right)$.
%
%Since the lattices contained in $\Zg^{tf,red}_{\Lambda}\backslash\left(x\pi^k(1-e_i)=0/x=0\right)$ are dense, \[\Zg^{tf,red}_{\Lambda}\backslash \left(x\pi^k(1-e_i)=0/x=0\right)\subseteq {_\Lambda}\Zg^{tf,red}\backslash d(\left(\pi^k(1-e_i)x=0/x=0\right)\cap {_\Lambda}\Zg^{tf,red}). \]
%
%So $d(\left(\pi^k(1-e_i)x=0/x=0\right)\cap {_\Lambda}\Zg^{tf,red})\subseteq \left(x\pi^k(1-e_i)=0/x=0\right).$
%
%The same argument gives that $d(\left(x\pi^k(1-e_i)=0/x=0\right)\cap \Zg_\Lambda^{tf,red})\subseteq \left(\pi^k(1-e_i)x=0/x=0\right).$
%
%So, since $d^2$ is the identity,  $d(\left(x\pi^k(1-e_i)=0/x=0\right)\cap \Zg^{tf,red}_{\Lambda})=\left(\pi^k(1-e_i)x=0/x=0\right)\cap \Zg^{tf,red}_{\Lambda}$.
%
\end{proof}

\begin{definition}\label{dfordedeking}
Let $U$ be an open subset of $\Zg_\Lambda^{tf}$. Define
\[dU:=\bigcup_{P\in\text{Max}R}d_P(U\cap\Zg_{\Lambda_P}^{rtf})\cup\bigcup_{S\in \lambda(U)}\mcal{V}(S^*)\] where $\lambda(U)$ is the set of all $Q\Lambda$-modules in $U$.

We will also use $d$ to denote the analogous map for open subsets of ${_\Lambda}\Zg^{tf}$.
\end{definition}

\begin{theorem}\label{dualitymain}
Let $R$ be a Dedekind domain, $Q$ its field of fractions and $\Lambda$ an $R$-order with $Q\Lambda$ a separable $Q$-algebra.

The mapping $d$ is an isomorphism $d$ between the lattice of open sets of $\Zg_\Lambda^{tf}$ and ${_\Lambda}\Zg^{tf}$ such that
\begin{enumerate}
\item if $L$ is an indecomposable right $\widehat{\Lambda_P}$-lattice then for all open sets $U\subseteq \Zg_\Lambda^{tf}$, $L\in U$ if and only if $L^{\dagger}\in dU$, and
\item for all open sets $U\subseteq \Zg_\Lambda^{tf}$, if $S$ is a simple $Q\Lambda$-module then  $S\in U$ if and only if $S^*\in dU$.
\end{enumerate}
\end{theorem}
\begin{proof}
Let $U$ be an open subset of $\Zg_\Lambda^{tf}$. We start by showing that for all open subsets $U\subseteq \Zg_\Lambda^{tf}$, $d^2U=U$.
So
\begin{eqnarray*}
  d^2U &=& d[\bigcup_{P}d_P(U\cap\Zg^{rtf}_{\Lambda_P})\cup\bigcup_{S\in \lambda(U)}\mcal{V}(S^*)] \\
   &=& \bigcup_{P}d_P[d_P(U\cap\Zg^{rtf}_{\Lambda_P})\cup\bigcup_{S\in U}\mcal{V}(S^*)\cap{_{\Lambda_P}}\Zg^{rtf}]\cup \bigcup_{S\in \lambda(U)}\mcal{V}(S)\\
   &=& d_P^2(U\cap\Zg^{rtf}_{\Lambda_P})\cup\bigcup_{P}\bigcup_{S\in \lambda(U)}d_P[\mcal{V}(S^*)\cap{_{\Lambda_P}}\Zg^{rtf}]\cup \bigcup_{S\in \lambda(U)}\mcal{V}(S)\\
   &=& \bigcup_{P}(U\cap\Zg^{rtf}_{\Lambda_P})\cup \bigcup_{S\in \lambda(U)}\mcal{V}(S) \\
   &=& U
\end{eqnarray*}

The first two equalities follow from the definition of $d$. The third is true because each $d_P$ is a lattice homomorphism. The fourth follows from \ref{dPonsimples} and the fifth follows from \ref{canformopenset}.

Thus $d$ gives a bijection between the lattice of open subsets of $\Zg_\Lambda^{tf}$ and ${_\Lambda}\Zg^{tf}$. We now just need to show that $d$ preserves inclusion.

Suppose $U\subseteq W$ are open subsets of $\Zg_\Lambda^{tf}$. Then $\lambda(U)\subseteq \lambda(W)$ and $U\cap \Zg^{rtf}_{\Lambda_P}\subseteq W\cap\Zg^{rtf}_{\Lambda_P}$ for all $P\in\text{Max}(R)$. So $d_P(U\cap\Zg^{rtf}_{\Lambda_P})\subseteq d_P(W\cap\Zg^{rtf}_{\Lambda_P})$ for all $P\in\text{Max}R$.  For all open sets $U$, $S\in \lambda(U)$ if and only if $S^*\in \lambda(dU)$. So $\lambda(U)\subseteq \lambda(W)$ implies $\lambda(dU)\subseteq\lambda(dW)$.  Therefore $dU\subseteq dW$.

Finally, $(1)$ holds for $d$ by \ref{lattopenisoonreducedpoints} and $(2)$ holds by definition of $d$.
\end{proof}

We finish this section with a different aspect of duality.

\begin{cor}\label{antiisoMar}
Let $R$ be a discrete valuation domain with maximal ideal generated by $\pi$. Let $k > k_0$. The lattices $[\pi|x,x=x]_{\Tf_\Lambda}$ and $[\pi|x,x=x]_{{_\Lambda}\Tf}$ are anti-isomorphic.
\end{cor}
\begin{proof}
Let $p=\pi+\pi^k\Lambda$. By \ref{clasMarlattiso}, $[\pi|x,x=x]_{\Tf_\Lambda}$ is isomorphic to $[p|x,x=x]_{\langle I\Tf_\Lambda\rangle}$ and $[\pi|x,x=x]_{{_\Lambda}\Tf}$ is isomorphic to $[p|x,x=x]_{\langle I{_\Lambda}\Tf\rangle}$. So, it is enough to show that $[p|x,x=x]_{\langle I\Tf_\Lambda\rangle}$ is anti-isomorphic to $[p|x,x=x]_{\langle I{_\Lambda}\Tf\rangle}$.

We have seen in \ref{leftimrightimmaranda} that $D\langle I\Tf_\Lambda\rangle=\langle I{_\Lambda}\Tf \rangle$. Thus Prest's duality for pp formulas, gives an anti-isomorphism between $\pp^1_{\Lambda_k}(\langle I\Tf_\Lambda\rangle)$ and ${_{\Lambda_k}}\pp^1(\langle I{_\Lambda}\Tf \rangle)$. Thus $[p|x,x=x]_{\langle I\Tf_\Lambda\rangle}$ is anti-isomorphic to $[x=0,px=0]_{\langle I{_\Lambda}\Tf \rangle}$.

Applying Goursat's lemma, \cite{Zieglermodules}, the formula $y=xp^{k-1}$ induces a lattice isomorphism between the intervals $[x=x,p^{k-1}x=0]_{\langle I{_\Lambda}\Tf \rangle}$ and $[y=0,p^{k-1}|y]_{\langle I{_\Lambda}\Tf \rangle}$. On $\langle I{_\Lambda}\Tf \rangle$, $p^{k-1}x=0$ is equivalent to $p|x$ and $p^{k-1}|y$ is equivalent to $py=0$. Thus $[x=0,px=0]_{\langle I{_\Lambda}\Tf \rangle}$ is isomorphic to $[p|x,x=x]_{\langle I{_\Lambda}\Tf \rangle}$.
\end{proof}

For the definition of the m-dimension of a modular lattice see \cite[Section 7.2]{PSL}.

\begin{cor}\label{mdimdual}
Suppose $R$ is a Dedekind domain with field of fractions $Q$, $\Lambda$ is an $R$-order and $Q\Lambda$ is separable. The m-dimension of $\pp_\Lambda^1(\Tf_{\Lambda})$ and ${_\Lambda}\pp^1({_\Lambda}\Tf)$ are equal.
\end{cor}
\begin{proof}
For each $P\in\text{Max}R$, by \cite[3.8]{TfpartZgorders}, the m-dimension of $\pp_{\Lambda_P}^1(\Tf_{\Lambda_P})$ is equal to the m-dimension of $[P|x,x=x]_{\Tf_{\Lambda_P}}$ plus $1$. Since $R_P$ is discrete valuation domain, by \ref{antiisoMar}, the m-dimension of $[P|x,x=x]_{\Tf_{\Lambda_P}}$ is equal to the m-dimension of $[P|x,x=x]_{{_{\Lambda_P}}\Tf}$. Thus, by \cite[3.8]{TfpartZgorders}, ${_{\Lambda_P}}\pp^1({_{\Lambda_P}}\Tf)$ has m-dimension equal to the m-dimension of $[P|x,x=x]_{\Tf_{\Lambda_P}}$ plus $1$ i.e. equal to the m-dimension of $\pp_{\Lambda_P}^1(\Tf_{\Lambda_P})$.

By \cite[3.9]{TfpartZgorders}, the m-dimension of $\pp_\Lambda^1(\Tf_{\Lambda})$ (respectively ${_\Lambda}\pp^1({_\Lambda}\Tf)$) is equal to the supremum of the m-dimensions of $\pp_{\Lambda_P}^1(\Tf_{\Lambda_P})$ (respectively ${_{\Lambda_P}}\pp^1({_{\Lambda_P}}\Tf)$) where $P\in\text{Max}R$.
\end{proof}

We now translate the above corollary into a result about the Krull-Gabriel dimensions of $(\Latt_\Lambda,\Ab)^{fp}$ and $({_\Lambda}\Latt,\Ab)^{fp}$. See \cite[2.1]{Geigle} for a definition of the Krull-Gabriel dimension of a (skeletally) small abelian category.

If $\mcal{C}\subseteq \mod\text{-}S$ is a covariantly finite subcategory then $(\mcal{C},\Ab)^{\text{fp}}$ is equivalent to $(\mod\text{-}S,\Ab)^{\text{fp}}/\mcal{S}(\mcal{C})$, the Serre localisation of $(\mod\text{-}S,\Ab)^{fp}$ at the Serre subcategory
\[\mcal{S}(\mcal{C}):=\{F\in (\mod\text{-}S,\Ab)^{fp} \st FC=0 \text{ for all } C\in\mcal{C}\}.\] See \cite{Herendringloccohfun} for details.

%Note that, \cite[11.1.33]{PSL}, $(\mod\text{-}S,\Ab)/\overrightarrow{\mcal{S}(\mcal{C})}$ is locally finitely presented and $(\mod\text{-}S,\Ab)^{\text{fp}}/\mcal{S}(\mcal{C})$ is the category of finitely presented objects of $(\mod\text{-}S,\Ab)/\overrightarrow{\mcal{S}(\mcal{C})}$ where $\overrightarrow{\mcal{S}(\mcal{C})}$ is the closure of $\mcal{S}(\mcal{C})$ under direct limits. Since both $(\mcal{C},\Ab)$ and $(\mod\text{-}S,\Ab)/\overrightarrow{\mcal{S}(\mcal{C})}$ are locally finitely presented, $(\mod\text{-}S,\Ab)^{\text{fp}}/\mcal{S}(\mcal{C})$ equivalent to $(\mcal{C},\Ab)^{\text{fp}}$ implies $(\mcal{C},\Ab)$ is equivalent to $(\mod\text{-}S,\Ab)/\overrightarrow{\mcal{S}(\mcal{C})}$. See \cite{CBlfp} for a proof that $(\mcal{C},\Ab)$ is locally finitely presented and for a proof that a locally finitely presented category is determined by its category of finitely presented objects.

By \cite[13.2.2]{PSL}, the Krull-Gabriel dimension of $(\mcal{C},\Ab)^{fp}/\mcal{S}(\mcal{C})$ is equal to the m-dimension of $\pp^1_S(\langle \mcal{C} \rangle)$.

Applying this to $\Latt_\Lambda$, which is a covariantly finite subcategory of $\mod\text{-}\Lambda$, we get that the Krull-Gabriel dimension of $(\Latt_\Lambda,\Ab)$ is equal to the m-dimension of $\pp_\Lambda^1\Tf_\Lambda$. Thus we get the following corollary to \ref{mdimdual}.

\begin{cor}\label{KGdual}
Suppose $R$ is a Dedekind domain with field of fractions $Q$, $\Lambda$ is an $R$-order and $Q\Lambda$ is separable. The Krull-Gabriel dimension of $(\Latt_\Lambda,\Ab)^{fp}$ is equal to the Krull-Gabriel dimension of $({_\Lambda}\Latt,\Ab)^{fp}$.
\end{cor}

\textbf{Acknowledgements}
\noindent
I would like to thank Carlo Toffalori for reading an early version of the results in sections \ref{SMarfun} and \ref{S-duality}, and more generally for support and encouragement during this project.

%\nocite{*}
\bibliographystyle{hplain}
\bibliography{Marandainitial}

\end{document}